\documentclass[twoside,11pt,a4paper]{article}
\usepackage{bbold}
\usepackage{bm}
\usepackage{bbm}
\usepackage{bbding}
\usepackage{latexsym}
\usepackage{amsmath}   
\usepackage{tipa}
\usepackage{mathrsfs}  
\usepackage{amssymb}   
\usepackage{amsfonts}
\usepackage{indentfirst}
\usepackage{amsbsy}
\usepackage{amsthm}
\usepackage{hyperref}
\def\beeq{\begin{equation}}      \def\eneq{\end{equation}}
\def\beeqy{\begin{eqnarray}}     \def\eneqy{\end{eqnarray}}
\def\bece{\begin{center}}        \def\ence{\end{center}}
\def\ba {\begin{array}}          \def\ea {\end{array}}
\def\bess{\begin{eqnarray*}}     \def\eess{\end{eqnarray*}}
\def\bes{\begin{split}}          \def\ens{\end{split}}
\def\bali{\begin{align}}         \def\enali{\end{align}}

\def\ali{\aligned}               \def\eali{\endaligned}

\def\pa{\partial} \def\ov{\overline} \def\disp{\displaystyle}\def\d{\displaystyle\frac}
\def\f{\frac}     \def\q {\quad}      
\def\la{\langle}  \def\ra{\rangle}   \def\l{\left} \def\r{\right}

\def\s1{\sqrt{-1}}

\def\suml{\sum\limits}

\theoremstyle{plain}
\newtheorem{theorem}{\noindent{\bf Theorem}}[section]
\newtheorem{definition}[theorem]{\noindent{\bf Definition}}
\newtheorem{prop}[theorem]{\noindent{\bf Proposition}}

\newtheorem{lem}[theorem]{\noindent{\bf Lemma}}

\newtheorem{ex}[theorem]{\noindent{\bf Example}}
\newtheorem{rem}[theorem]{\noindent{\bf Remark}}

\def\az{\alpha}      \def\bz{\beta}     \def\gz{\gamma}     \def\dz{\delta}                   \def\tz{\theta}                  \def\lz{\lambda}         \def\rz{\rho}          \def\sz{\sigma}                        \def\oz{\omega}

      \def\ttz{\Theta}  
    \def\ooz{\Omega}

\def\CA{{\cal A}}    
    
  \def\CH{{\cal H}}  
    \def\CL{{\cal L}}

\def\CS{{\cal S}}    
\def\CV{{\cal V}}    \def\CX{{\cal X}}

    \def\BC{\mathbb{C}}

\def\BP{{\Bbb P}}

     \def\det {{\rm det\,}}

        \def\bin#1#2 {{#1\choose#2}}

\def\disp {\displaystyle}
           \def\dfrac#1#2 {{\displaystyle{#1\over#2}}}
           \def\din#1#2 {{\displaystyle{#1\choose#2}}}

\def\disp{\displaystyle}

\def\a.{{\rm \ddot{a}}}

\begin{document}
\renewcommand{\thesection}{\arabic{section}}
\renewcommand{\theequation}{\thesection.\arabic{equation}}
\baselineskip=18pt

\title{\bf Characterizations of complex Finsler Metrics}
\author{ Hongjun Li$^{1}$, Hongchuan Xia$^{2}$\\
\emph{\small $^1$School of Mathematics and Statistics, Henan University, Kaifeng {\rm 475004}, China;}\\
\emph{\small $^2$School of Mathematics and Statistics,
Xinyang Normal University, Xinyang {\rm 464000}, China.}\\
\emph{\small \rm Email: lihj@vip.henu.edu.cn, xhc@xynu.edu.cn}
 }
\date{}
\maketitle

\noindent{\bf Abstract}\hskip3mm
Munteanu \cite{Mu} defined the canonical connection associated to a strongly pseudoconvex complex Finsler manifold $(M,F)$.
We first prove that the holomorphic sectional curvature tensors of the canonical connection coincide with those of the Chern-Finsler connection associated to $F$ if and only if $F$ is a K\"ahler-Finsler metric.
We also investigate the relationship of the Ricci curvatures (resp. scalar curvatures) of these two connections when $M$ is compact. As an application, two characterizations of balanced complex Finsler metrics are given. Next, we obtain a sufficient and necessary condition for a balanced complex Finsler metric to be K\"ahler-Finsler.
Finally, we investigate conformal transformations of a balanced complex Finsler metric.

\vspace*{2mm}

\vspace*{2mm}

\vspace*{4mm} \thispagestyle{empty}
\section{Introduction}
%
%
%

Complex Finsler manifolds are complex manifolds endowed with complex Finsler metrics, which are
more general than Hermitian manifolds. Recently, complex Finsler geometry has attracted much interest because a number of complex Finsler metrics play an important role in geometric function theory of holomorphic mappings \cite{AP}.

In complex Finsler geometry, the Chern-Finsler connection associated to a strongly pseudoconvex complex Finsler metric enjoys some beautiful features as the Hermitian (aka Chern) connection associated to a Hermitian metric. According to the vanishing of some parts of the torsion of the Chern-Finsler connection, there are three kinds of metrics in complex Finsler geometry, which are the extension of K\"ahler metrics, and  respectively,  called the strongly K\"ahler-Finsler, K\"ahler-Finsler and weakly K\"ahler-Finsler \cite{AP}. In fact, K\"ahler-Finsler metrics are actually strongly K\"ahler \cite{CS}. Hence, there are only two kinds of K\"ahler-Finsler metrics with respect to the Chern-Finsler connection in complex Finsler geometry.

Let $F$ be a strongly pseudoconvex complex Finsler metric on a complex manifold $M$, and $D$ be the Chern-Finsler connection associated to $F$. The Chern-Finsler connection $D$ can induce a canonical connection $\nabla$ \cite{Mu}. Let $R_{\az\bar\bz\mu\bar\nu}$ and $\ooz_{\az\bar\bz;\mu\bar\nu}$ denote the holomorphic sectional curvature tensors of the canonical connection $\nabla$ and the Chern-Finsler connection $D$, respectively,  where lowercase Greek indices run from 1 to $n={\rm dim}_{\BC}(M)$.
When $F$ is a K\"ahler-Finsler metric, one has $R_{\az\bar\bz\mu\bar\nu}=\ooz_{\az\bar\bz;\mu\bar\nu}$. A naive question is, when a strongly pseudoconvex complex Finsler metric $F$ satisfies $R_{\az\bar\bz\mu\bar\nu}=\ooz_{\az\bar\bz;\mu\bar\nu}$, whether it is K\"ahler-Finsler. In this paper, we give a positive answer to this question.
Denote by $K_{\nabla}(H)$ and $K_D(H)$ (resp. $K_{\nabla}(v)$ and $K_D(v)$) the holomorphic flag curvatures (resp. the holomorphic sectional curvatures) of the canonical connection and the Chern-Finsler connection along a horizontal vector $H\in\CH$ (resp. a holomorphic vector $v$), respectively.
We have the following result:

\begin{theorem}
Let $(M,F)$ be a strongly pseudoconvex complex Finsler manifold. Then
\begin{itemize}
\item[{\rm (1)}] $R_{\az\bar\bz\mu\bar\nu}=\ooz_{\az\bar\bz;\mu\bar\nu}$, if and only if
$K_{\nabla}(H)=K_D(H)$, if and only if $F$ is a K\"ahler-Finsler metric;
\item[{\rm (2)}] $K_{\nabla}(v)=K_{D}(v)$,
if and only if $F$ is a weakly K\"ahler-Finsler metric.
\end{itemize}
\end{theorem}

Let $z=(z^{1},\cdots,z^{n})$ be a local holomorphic coordinate system on $M$, and
$ {G}_{\az\bar{\bz}}$ be the fundamental tensor of $(M,{F})$. Set $(G^{\bar{\bz}\az})=(G_{\az\bar{\bz}})^{-1}$.
We denote by
$$
\mathbf{Ric}_{\nabla} =\sqrt{-1}R _{\mu\bar{\nu}}\mathrm{d}z^{\mu}\wedge \mathrm{d}\bar{z}^{\nu}\q
with \q R _{\mu\bar{\nu}}=G^{\bar{\bz}\az}R_{\az\bar\bz\mu\bar\nu},
$$
and
$$
\mathbf{Ric}_{D} =\sqrt{-1}\ooz _{\mu\bar{\nu}}\mathrm{d}z^{\mu}\wedge \mathrm{d}\bar{z}^{\nu}\q
with \q
\ooz _{\mu\bar{\nu}}
=G^{\bar{\bz}\az}\ooz_{\az\bar\bz;\mu\bar\nu}
$$
the Ricci curvatures of the canonical connection $\nabla$ and the Chern-Finsler connection $D$, respectively.
The scalar curvatures of the canonical connection $\nabla$ and the Chern-Finsler connection $D$ are denoted
by $s_{\nabla}=G^{\bar{\nu}\mu}R_{\mu\bar{\nu}}$ and $s_{D}=G^{\bar{\nu}\mu}\ooz_{\mu\bar{\nu}}$, respectively.
When $M$ is compact, we obtain explicit relations between these two kinds of Ricci curvatures by using the horizontal fundamental form $\oz_{\CH}$ and furthermore  give two characterizations
of balanced complex Finsler metrics. Here we say that a compact strongly pseudoconvex complex Finsler manifold $(M,F)$ is balanced if the horizontal fundamental form $\oz_{\CH}=\sqrt{-1} G_{\az\bar{\bz}}\mathrm{d}z^{\az}\wedge \mathrm{d}\bar{z}^{\bz}$ is horizonal co-closed, i.e., $\mathrm{d}_{\CH}^*\oz_{\CH}=0$.

\begin{theorem}
Let $(M,F)$ be a compact strongly pseudoconvex complex Finsler manifold. Then
\begin{equation}\label{ric}
\mathbf{Ric}_{D} -\mathbf{Ric}_{\nabla} =\d{1}{2}\l({\pa}_{\CH}{\pa}_{\CH}^*\oz_{\CH}
+\bar{\pa}_{\CH}\bar{\pa}_{\CH}^*\oz_{\CH}\r),
\end{equation}
\begin{equation}\label{sca}
s _{D}-s _{\nabla}=\d{1}{2}\l\la{\pa}_{\CH}{\pa}_{\CH}^*\oz_{\CH}
+\bar{\pa}_{\CH}\bar{\pa}_{\CH}^*\oz_{\CH},\oz_{\CH}\r\ra_{\BP\tilde{M}}.
\end{equation}
Hence, the following three statements are equivalent:
\begin{itemize}
\item[{\rm (a)}] $F$ is a balanced complex Finsler metric;
\item[{\rm (b)}] $\mathbf{Ric}_{D} =\mathbf{Ric}_{\nabla} $;
\item[{\rm (c)}] $\disp\int_{\BP\tilde{M}}s_{D}\mathrm{d}\mu_{\BP\tilde{M}}
=\disp\int_{\BP\tilde{M}}s_{\nabla}\mathrm{d}\mu_{\BP\tilde{M}}$.
\end{itemize}
\end{theorem}

We can extend the canonical connection $\nabla$ to $T_{\BC}\l(T^{1,0}M\backslash \{0\}\r)$, denoted by $\nabla^{\mathbbm{c}}$. The connection $\nabla^{\mathbbm{c}}$ is a natural generalization of the complexified Levi-Civita connection on the complexified tangent bundle of a Hermitian manifold.
Denote by  $\mathbf{Ric} _{\nabla^{\mathbbm{c}}}=\sqrt{-1}K _{\mu\bar{\nu}}\mathrm{d}z^\mu\wedge \mathrm{d}\bar{z}^{\nu}$ and $s _{\nabla^{\mathbbm{c}}}=G^{\bar{\nu}\mu}K _{\mu\bar{\nu}}$ the Ricci curvature and scalar curvature of the connection $\nabla^{\mathbbm{c}}$, respectively. We have

\begin{theorem}
Let $(M,F)$ be a compact strongly pseudoconvex complex Finsler manifold. Then
\begin{equation}\label{ric2}
\mathbf{Ric}_{D} -\mathbf{Ric}_{\nabla^{\mathbbm{c}}} =\d{1}{2}\l({\pa}_{\CH}{\pa}_{\CH}^*\oz_{\CH}
+\bar{\pa}_{\CH}\bar{\pa}_{\CH}^*\oz_{\CH}\r)+\sqrt{-1}\mathfrak{S}\circ \ov{\mathfrak{S}},
\end{equation}
\begin{equation}\label{sca2}
s_{D}-s_{\nabla^{\mathbbm{c}}}=\d{1}{2}\l\la{\pa}_{\CH}{\pa}_{\CH}^*\oz_{\CH}
+\bar{\pa}_{\CH}\bar{\pa}_{\CH}^*\oz_{\CH},\oz_{\CH}\r\ra_{\BP\tilde{M}}
+\la \mathfrak{S},\mathfrak{S}\ra_{\BP\tilde{M}},
\end{equation}
where $\mathfrak{S}$ is given by \eqref{s}.
Hence, the following four statements are equivalent:
\begin{itemize}
\item[{\rm (a)}] $F$ is a K\"ahler-Finsler metric;
\item[{\rm (b)}] $\mathbf{Ric}_{D} =\mathbf{Ric}_{\nabla^{\mathbbm{c}} } $;
\item[{\rm (c)}] $\disp\int_{\BP\tilde{M}}s_{\nabla}\mathrm{d}\mu_{\BP\tilde{M}}
=\disp\int_{\BP\tilde{M}}s_{\nabla^{\mathbbm{c}}}\mathrm{d}\mu_{\BP\tilde{M}}$;
\item[{\rm (d)}] $\disp\int_{\BP\tilde{M}}s_{D}\mathrm{d}\mu_{\BP\tilde{M}}
=\disp\int_{\BP\tilde{M}}s_{\nabla^{\mathbbm{c}}}\mathrm{d}\mu_{\BP\tilde{M}}$.
\end{itemize}
\end{theorem}

Denote by $\hat{\ooz}^\az_{\bz;\mu\bar{\nu}}\mathrm{d}z^\mu\wedge \mathrm{d}{\bar z}^\nu$ the horizontal part of the curvature form of the complex Rund connection $\hat{D}$. By analogy with the K\"ahler-like metric defined in \cite{YZ}, we say $F$ is a Rund K\"ahler-Finsler-like metric if $\hat{\ooz}^\az_{\bz;\mu\bar{\nu}}
=\hat{\ooz}^\az_{\mu;\bz\bar{\nu}}$. When the corresponding manifold is compact, these metrics are more special than balanced complex
Finsler metrics.

\begin{theorem}
Let $(M,F)$ be a Rund K\"ahler-Finsler-like manifold. If $M$ is compact, then $F$ is a balanced complex Finsler metric.
\end{theorem}

We obtain a sufficient and necessary condition
for a balanced complex Finsler metric to be K\"ahler-Finsler.

\begin{theorem}
Let $(M,F)$ be a balanced complex Finsler manifold. Then $F$ is a K\"ahler-Finsler metric, if and only if
$$
\l\la\pa_{\CH}\bar{\pa}_{\CH}\oz_{\CH}, \oz_{\CH}^2\r\ra_{\BP\tilde{M}}=0.
$$
\end{theorem}

In the end, we investigate conformal transformations of balanced complex Finsler metrics.
We find that balanced complex Finsler metrics are clearly unique within each conformal class.
In the non-compact case, we get a sufficient and necessary condition
for two Rund K\"ahler-Finsler metrics to be conformal.

\begin{theorem}
On a compact complex manifold $M$, there exists at most one balanced complex Finsler metric (up to constant multiples) in each conformal class of strongly pseudoconvex complex Finsler metrics. For a non-compact strongly pseudoconvex Rund K\"ahler-Finsler-like manifold $(M,F)$,  the conformal metric
$\tilde {F}(z,v)={\rm e}^{\frac{1}{2}\rz(z)}F(z,v)$ is Rund K\"ahler-Finsler-like if and only if $\pa\bar{\pa}\rho=0$.
\end{theorem}

We investigate the characterization for  a strongly pseudoconvex complex Finsler metric to be conformal to a balanced complex Finsler metric.

\begin{theorem}
On a compact strongly pseudoconvex complex Finsler manifold $(M,F)$, the conformal metric $\tilde {F}(z,v)={\rm e}^{\frac{1}{2}\rz(z)}F(z,v)$ is balanced, if and only if
\beeq
\mathbf{Ric}_{D} -\mathbf{Ric}_{\nabla} =(n-1)\sqrt{-1}\pa\bar{\pa}\rho,
\eneq
if and only if
\beeq
\int_{\BP\tilde{M}}\l(s_{D} -s_{\nabla} \r)\mathrm{d}\mu_{\BP\tilde{M}}
=(n-1)\l(\sqrt{-1}\pa\bar{\pa}\rho,\oz_{\CH}\r)_{\BP\tilde{M}}.
\eneq
\end{theorem}

We also study the Chern-Finsler total scalar curvature under conformal transformations of a balanced complex Finsler manifold.

\begin{theorem}
Let $(M,F)$ be a balanced complex Finsler manifold with positive Chern-Finsler scalar curvature. Then the conformal metric $\tilde {F}(z,v)={\rm e}^{\frac{1}{2}\rz(z)}F(z,v)$ owns positive Chern-Finsler total scalar curvature, that is
\beeq\label{ }
\int_{\BP\tilde{M}}\tilde{s}_{\tilde{D}}  \mathrm{d}\tilde{\mu}_{\BP\tilde{M}}>0.
\eneq
\end{theorem}

The rest of this paper is arranged as follows.
In Section 2, we provide a brief overview of complex Finsler manifolds, and introduce balanced complex Finsler manifolds and Rund K\"ahler-Finsler-like manifolds.
In Section 3, we recall detailedly the canonical connection defined by Munteanu.
In Section 4, we examine the behavior of holomorphic sectional curvature tensors of the canonical connection and the Chern-Finsler connection, and also investigate the corresponding Ricci curvatures and scalar curvatures.
In Section 5, we prove Theorems 1.4 and 1.5, and give two examples of balanced complex Finsler metrics (resp. Rund K\"ahler-Finsler-like metrics). In Section 6, we investigate conformal transformations of
a balanced complex Finsler metric.

%
%
%
\section{Preliminaries}
\setcounter{equation}{0}
%
%
%
\subsection{Complex Finsler manifolds}

Let $M$ be a complex manifold with ${\rm dim}_{\BC}(M)=n\geq2$.
We shall denote by $\pi:T^{1,0}M\rightarrow M$ the holomorphic tangent bundle of $M$.
Suppose that $z=(z^{1},\cdots,z^{n})$ is a local holomorphic coordinate system on $M$, then an arbitrary holomorphic vector $v\in T^{1,0}M$ can be written as
$$
v=v^\az\frac{\pa}{\pa z^{\az}}.
$$
So $(z,v)$ is a local holomorphic coordinate system on $T^{1,0}M$.
Let $\tilde{M}=T^{1,0}M\backslash \{0\}$ denote the slit holomorphic tangent bundle.

\begin{definition}{\rm \cite{AP}}
A complex Finsler metric on a
complex manifold $M$ is a continuous function
$F:\, T^{1,0}M\rightarrow[0,+\infty)$ satisfying:
\begin{itemize}
\item[{\rm (a)}] $G=F^2$ is smooth on $\tilde{M}$;
\item[{\rm (b)}] $F(z,v)>0$ for all $(z,v)\in \tilde{M}$;
\item[{\rm (c)}] $F(z,\zeta v)=|\zeta|F(z,v)$ for all $(z,v)\in T^{1,0}M$ and $\zeta\in\mathbb{C}$.
\end{itemize}
\end{definition}

For simplicity, we denote $\ov{z^{\bz}}$, $\ov{v^{\bz}}$
by $ \bar{z}^{\bz}$, $ \bar{v}^{\bz}$, the derivatives of $G$ with respect to $v$ by
$$
G_{\az}=\dot{\pa}_{\az}G=\d{\pa G}{\pa v^\az},\q G_{\bar\az}=\dot{\pa}_{\bar\az}G=\d{\pa G}{\pa \bar{v}^\az},\q
G_{\az\bar{\bz}}=\dot\pa_{\az}\dot\pa_{\bar{\bz}}G
=\d{\pa^2 G}{\pa v^\az \pa \bar{v}^\bz},
$$
and the derivatives of $G$ with respect to $z$ by indices after a semicolon,
for instance
$$
G_{;\az}=\pa_{\az}G=\d{\pa G}{\pa z^\az}, \q
G_{;\az\bar{\bz}}=\pa_{\az}\pa_{\bar\bz}G
=\d{\pa^2 G}{\pa z^{\az}\pa \bar{z}^{\bz}}, \q
G_{\az;\bar{\bz}}=\dot{\pa}_{\az}\pa_{\bar\bz}G
=\d{\pa^2 G}{\pa \bar{z}^{\bz}\pa v^{\az}}.
$$

\begin{definition}{\normalfont \cite{AP}}
A complex Finsler metric $F$ will be said strongly pseudoconvex
if the Levi matrix $\l(G_{\az\bar{\bz}}\r)$ is positive definite on $\tilde{M}$.
\end{definition}

Suppose $(M,F)$ is a strongly pseudoconvex complex Finsler manifold.
Let us denote by
\begin{equation}
\varGamma^{\az}_{;\mu}=G^{\bar{\lz}\az}G_{\bar{\lz};\mu}
\end{equation}
the Chern-Finsler nonlinear connection coefficients.
Set
\begin{equation}
\delta_{\mu}=\partial_{\mu}-\varGamma_{;\mu}^{\az}\dot{\partial}_{\az},\q
\delta v^{\az}={\rm d}v^{\az}+\varGamma_{;\bz}^{\az}{\rm d}z^{\bz},
\end{equation}
then the holomorphic tangent bundle $T^{1,0}\tilde{M}$ of $\tilde{M}$ can
be decomposed into the direct sum of the horizontal bundle $\mathcal {H}$ spanned by $\{\dz_{\mu}\}$ and the vertical bundle $\CV$ spanned by $\{\dot\pa_{\az}\}$.
The dual frame of $\{\dz_\mu,\dot\pa_\az\}$ is $\{\mathrm{d}z^{\mu},\delta v^{\az}\}$.
In other words, we have
$$
\ali
T^{1,0}\tilde{M}&=\CH\oplus\CV
={\rm span}\{\dz_{\mu}\}\oplus {\rm span}\{\dot\pa_{\az}\},\\
T^{*1,0}\tilde{M}&=\CH^*\oplus\CV^*
={\rm span}\{\mathrm{d}z^{\mu}\}\oplus {\rm span}\{\dz v^{\az}\}.
\eali
$$
We can define a Hermitian metric $\la \, , \, \ra$ on the vertical bundle $\CV$ by the strong pseudoconvexity of $F$. For any $ v\in\tilde{M}_z$, we set
$$
\l\la Z , W \r\ra_v=G_{\az\bar{\bz}}(z,v)Z^\az\bar{W}^{\bz}, \q
Z=Z^\az\dot{\pa}_{\az}, \, W=W^\az\dot{\pa}_{\az}\in \CV_v.
$$
There exists a unique complex vertical connection $D:\CX(\CV)\rightarrow \CX(T_{\BC}^{*}{\tilde M}\otimes \CV)$
such that
\begin{equation*}
X\l\la Z, W \r\ra =\l\la \nabla_{X}Z, W \r\ra+ \l\la Z, \nabla_{\bar{X}}W \r\ra
\end{equation*}
for all $X\in T^{1,0}\tilde M$ and $Z,W \in \CX(\CV)$.
Kobayashi \cite{Ko} first introduced this connection $D$, Abate and Patrizio \cite{AP} called it the Chern-Finsler connection. The connection 1-forms of $D$ are given by
\begin{equation}
\tz_{\bz}^{\az}=G^{\bar{\lz}\az}\partial G_{\bz\bar{\lz}}
=\varGamma_{\bz;\mu}^{\az}\mathrm{d}z^{\mu}+\varGamma_{\bz\gz}^{\az}\delta v^{\gz},
\end{equation}
where
\begin{equation}
 \varGamma_{\bz;\mu}^{\az}=G^{\bar{\lz}\az}\dz_{\mu}\l(G_{\bz\bar{\lz}}\r),\quad
 \varGamma_{\bz\gz}^{\az}=G^{\bar{\lz}\az}G_{\bz\bar{\lz}\gz}.
\end{equation}

Let $\ttz :\, \CV \rightarrow \CH$ be a complex horizontal map locally defined by $\ttz(\dot\pa_{\az}) = \dz_{\az}$.
We can transfer the Hermitian structure $\la ~, ~\ra$ on $\CH$ just by setting
$$
\l\la H , K \r\ra_v=\l\la \ttz^{-1}\l(H\r),\ttz^{-1}\l(K\r)\r\ra_v, \q
\forall H, \; K\in \CH_v.
$$
Furthermore, there is a Sasaki--type lift Hermitian structure \cite{Mu}
\beeq\label{stl}
\tilde{G}= G_{\az\bar{\bz}}\mathrm{d}z^{\az} \otimes \mathrm{d}\bar{z}^{\bz}+ G_{\az\bar{\bz}} \dz v^{\az} \otimes \dz\bar{v}^{\bz}
\eneq
on $\tilde{M}$. We can also extend $D$ to a complex linear connection on $\CH$ by setting
\begin{equation}
D_{X}H=\ttz(D_X(\ttz^{-1}H)), \q H\in \CX(\CH),\ \ X\in T_{\BC}\tilde{M}.
\end{equation}

The $(2,0)$-torsion and $(1,1)$-torsion of the Chern-Finsler connection are denoted by $\tz=\tz^{\mu}\otimes\dz_{\mu}$ and $\tau=\tau^{\az}\otimes \dot{\pa}_{\az}$ respectively, where
\begin{equation}\label{te}
\begin{split}
\tz^{\mu}&=\d{1}{2}\l(\varGamma^{\mu}_{\nu;\sz}-\varGamma^{\mu}_{\sz;\nu}\r)
\mathrm{d}z^{\sz}\wedge \mathrm{d}z^{\nu}+\varGamma^{\mu}_{\nu\gz}\dz v^{\gz}\wedge \mathrm{d}z^{\nu},\\
\tau^{\az}&=-\dz_{\bar{\nu}}\l(\varGamma^{\az}_{;\mu}\r)
\mathrm{d}z^{\mu}\wedge \mathrm{d}\bar{z}^{\nu}-\dot{\pa}_{\bar{\bz}}\l(\varGamma^{\az}_{;\mu}\r)
\mathrm{d}z^{\mu}\wedge \dz\bar{z}^{\nu}.
\end{split}
\end{equation}
\begin{definition}{\normalfont\cite{AP}}
We call $F$ strongly K\"ahler if $\tz(H,K)=0, \, \forall H,K\in \CH$;
K\"ahler if $\tz(H,\chi)=0, \, \forall H\in \CH$; weakly K\"ahler if $\la\tz(H,\chi),\chi\ra=0,
\, \forall H\in \CH$, where $\chi=v^\alpha\delta_\alpha$ is the horizontal radical vector field.
\end{definition}

In local coordinates, $F$ is strongly K\"ahler if and only if
$\varGamma_{\bz;\mu}^{\az}=\varGamma_{\mu;\bz}^{\az}$;
is K\"ahler if and only if $\big(\varGamma_{\bz;\mu}^{\az}-\varGamma_{\mu;\bz}^{\az}\big)v^{\bz}=0$;
is weakly K\"ahler if and only if $G_{\az}\big(\varGamma_{\bz;\mu}^{\az}-\varGamma_{\mu;\bz}^{\az}\big)
v^{\bz}=0$. Chen and Shen \cite{CS} proved K\"ahler-Finsler metrics are actually strongly K\"ahler.

The curvature operator of the Chern-Finsler connection $D$ is given by
\begin{equation}
\Omega=\Omega^{\az}_{\bz}\otimes \big(\mathrm{d}z^{\bz}\otimes\dz_{\az}
+\dz{v}^{\bz}\otimes\dot\pa_{\az}\big),
\end{equation}
where $\Omega_{\bz}^{\az}=\bar\pa\tz_{\bz}^{\az}.$
In local coordinates, $\Omega_{\bz}^{\az}$ can be decomposed as
\begin{equation}
\begin{split}
\Omega_{\bz}^{\az}
=&\Omega^{\az}_{\bz;\mu\bar{\nu}}\mathrm{d}z^{\mu}\wedge \mathrm{d}\bar{z}^{\nu}
+\Omega^{\az}_{\bz\mu;\bar{\nu}}\dz v^{\mu}\wedge \mathrm{d}\bar{z}^{\mu}
+\Omega^{\az}_{\bz\bar{\nu};\mu}\mathrm{d}z^{\mu}\wedge\dz\bar{v}^{\nu}
+\Omega^{\az}_{\bz\mu\bar{\nu}}\dz v^{\mu}\wedge\dz \bar{v}^{\nu},
\end{split}
\end{equation}
where
\begin{equation}
\begin{split}
&\Omega^{\az}_{\bz;\mu\bar{\nu}}=-\dz_{\bar{\nu}}\l(\varGamma^{\az}_{\bz;\mu}\r)
-\varGamma^{\az}_{\bz \gz}\dz_{\bar\nu}\l(\varGamma^{\gz}_{;\mu}\r), \ \
\Omega^{\az}_{\bz \mu;\bar \nu}= -\dz_{\bar\nu}\l(\varGamma^{\az}_{\bz \mu}\r),\\
&\Omega^{\az}_{\bz\bar{\nu};\mu}=-\dot{\pa}_{\bar\nu}\l(\varGamma^{\az}_{\bz;\mu}\r)
-\varGamma^{\az}_{\bz\gz}\dot\pa_{\bar \nu}\l(\varGamma^{\gz}_{;\mu}\r), \ \
\Omega^{\az}_{\bz\mu\bar{\nu}}=-\dot\pa_{\bar \nu}\l(\varGamma^{\az}_{\bz\mu}\r).
\end{split}
\end{equation}
We denote by $\Omega_{\az\bar{\bz};\mu\bar{\nu}}=G_{\gz\bar{\bz}}\Omega^{\gz}_{\az;\mu\bar{\nu}}$ the holomorphic sectional curvature tensors of $D$.
Locally,
\begin{equation}
\Omega_{\az\bar{\bz};\mu\bar{\nu}}
=-\dz_{\bar{\nu}}\dz_{\mu}\l(G_{\az\bar{\bz}}\r)+\dz_{\mu}\l(G_{\az\bar{\lz}}\r)
G^{\bar{\lz}\kappa}\dz_{\bar{\nu}}\l(G_{\kappa\bar{\bz}}\r)
-G_{\az\bar{\bz}\sz}\dz_{\bar\nu}\l(\varGamma^{\sz}_{;\mu}\r).
\end{equation}
The horizontal holomorphic flag curvature $K_{D}(H)$ of $F$ along a horizontal vector $H=H^{\az} \dz_{\az}\in\CH_v$ is defined by
\begin{equation}
K_{D}(H)=\d{\la\ooz(H,\bar{H})H,H\ra_v}{\la H,H\ra_v}.
\end{equation}
The holomorphic sectional curvature $K_{D}(v)$ of $F$ along $v\in \tilde{M}$ is defined by
\begin{equation}
K_{D}(v)=\d{\la \Omega\l(\chi,\bar{\chi}\r)\chi,\chi\ra_v}{G(v)^2}.
\end{equation}
The Ricci curvatures and scalar curvatures of the Chern-Finsler connection $D$ are, respectively, given by
\beeq
\mathbf{Ric}_{D} =\sqrt{-1}\ooz _{\mu\bar{\nu}}\mathrm{d}z^{\mu}\wedge \mathrm{d}\bar{z}^{\nu}\q
with \q \ooz _{\mu\bar{\nu}}=G^{\bar{\bz}\az}\ooz_{\az\bar\bz;\mu\bar\nu},
\eneq
and
\beeq
s_{D} =G^{\bar{\bz}\az}G^{\bar{\nu}\mu}\ooz_{\az\bar\bz;\mu\bar\nu}.
\eneq

\subsection{Balanced complex Finsler metrics}

A strongly pseudoconvex complex Finsler metric $F$
can induce a Hermitian metric 
$$
\tilde{G}= G_{\az\bar{\bz}}\mathrm{d}z^{\az} \otimes \mathrm{d}\bar{z}^{\bz}
+ (\log G)_{\az\bar{\bz}} \dz v^{\az} \otimes \dz\bar{v}^{\bz}.
$$
on the projective tangent bundle $\BP\tilde{M}=\tilde{M}/(\BC\backslash \{0\})$.
If we set $\oz_\CV=\sqrt{-1}(\log G)_{\az\bar{\bz}}\dz v^\az\wedge \dz\bar{v}^{\bz}$ and $\oz_\CH=\sqrt{-1}G_{\az\bar{\bz}}\mathrm{d}z^\az\wedge \mathrm{d}\bar{z}^\bz$,
then the invariant volume form of $\BP\tilde{M}$ is \cite{ZZ}
\begin{equation}
\mathrm{d}\mu_{\BP\tilde{M}}=\d{\oz_{\CV}^{n-1}}{(n-1)!}\wedge \d{\oz_{\CH}^n}{n!}.
\end{equation}

Let $\CA^{p,q}$ be the set of smooth horizontal $(p,q)$-forms on $\BP\tilde{M}$, that is, all coefficients of any $\varphi\in\CA^{p,q}$ are zero homogeneous with respect to fibre coordinates $v$. In local coordinates,
$$
\varphi=\d{1}{p!q!}\varphi_{\az_1\cdots\az_p\bar{\bz}_1\cdots\bar{\bz}_q}\mathrm{d}z^{\az_1}\wedge\cdots\wedge \mathrm{d}z^{\az_p}\wedge \mathrm{d}\bar{z}^{\bz_1}\wedge\cdots\wedge \mathrm{d}\bar{z}^{\bz_q}.
$$
If $M$ is compact, then for any $\varphi,\, \psi\in\CA^{p,q}$, there is a natural Hermitian inner product on $\CA^{p,q}$ given by
\beeq
(\varphi, \psi)_{\BP\tilde{M}}=\int_{\BP\tilde{M}}\la \varphi, \psi \ra_{\BP\tilde{M}}
\mathrm{d}\mu_{\BP\tilde{M}},
\eneq
where
\beeq
\la \varphi, \psi \ra_{\BP\tilde{M}}
=\d{1}{p!q!}\varphi_{\az_1\cdots\az_p\bar{\bz}_1\cdots\bar{\bz}_q}
G^{\bar{\gz}_1\az_1}\cdots G^{\bar{\gz}_p\az_p}
G^{\bar{\bz}_1\dz_1}\cdots G^{\bar{\bz}_q\dz_q}\ov{\psi_{\gz_1\cdots\gz_p\bar{\dz}_1\cdots\bar{\dz}_q}}.
\eneq

Let $\varphi\in\CA^{p,q}$, then two horizonal operator ${\pa}_{\CH}: \CA^{p,q}\rightarrow\CA^{p+1,q}$ and $\bar{\pa}_{\CH}: \CA^{p,q}\rightarrow\CA^{p,q+1}$
are defined by
\[
\ali
{\pa}_{\CH}\varphi&=\d{1}{(p+1)!q!}\suml_{j=1}^{p+1}(-1)^{j-1}\dz_{{\az}_j}
\l(\varphi_{\az_1\cdots\widehat{{\az}_j}\cdots\az_{p+1}\bar{\bz}_1\cdots\bar{\bz}_{q}}\r)
\mathrm{d}z^{\az_1}\wedge\cdots\wedge \mathrm{d}z^{\az_{p+1}}\wedge \mathrm{d}\bar{z}^{\bz_1}\wedge\cdots\wedge \mathrm{d}\bar{z}^{\bz_{q}},\\
\bar{\pa}_{\CH}\varphi&=\d{1}{p!(q+1)!}\suml_{j=1}^{q+1}(-1)^{p+j-1}\dz_{\bar{\bz}_j}
\l(\varphi_{\az_1\cdots\az_p\bar{\bz}_1\cdots\widehat{\bar{\bz}_j}\cdots\bar{\bz}_{q+1}}\r)
\mathrm{d}z^{\az_1}\wedge\cdots\wedge \mathrm{d}z^{\az_p}\wedge \mathrm{d}\bar{z}^{\bz_1}\wedge\cdots\wedge \mathrm{d}\bar{z}^{\bz_{q+1}},
\eali
\]
where $\az_1\cdots\widehat{{\az}_j}\cdots\az_{p+1}$ is the sequence which we get from $\az_1\cdots\az_{p+1}$ by suppressing ${\az}_j$.
Let us denote by $\mathrm{d}_{\CH}=\pa_{\CH} +\bar{\pa}_{\CH}$. We can define the formal adjoint operator $\bar{\pa}_{\CH}^*$ (resp. ${\pa}_{\CH}^*$) of $\bar{\pa}_{\CH}$ (resp. ${\pa}_{\CH}$) with respect to
the inter product $(\,,\, )_{\BP\tilde{M}}$. In \cite{XZQ}, the authors defined the Hodge star operator $*$
to better represent adjoint operators ${\pa}_{\CH}^*$ and $\bar{\pa}_{\CH}^*$. The operator
$*: \CA^{p,q} \rightarrow \CA^{n-q,n-p}$ is defined by the relation
\beeq	
\int_{\BP\tilde{M}}\varphi \wedge *\bar{\psi} \wedge \d{\oz_{\CV}^{n-1}}{(n-1)!}
= (\varphi,\psi)_{\BP\tilde{M}},\q \varphi,\,\psi \in \CA^{p,q}.
\eneq
Then \cite{XZQ}
\begin{itemize}
\item[{\rm (1)}] $\ov{*\psi}=*\bar{\psi}$;
\item[{\rm (2)}] $**\psi=(-1)^{p+q}\psi$;
\item[{\rm (3)}] ${\pa}_{\CH}^*=-*\bar{\pa}_{\CH} *$, and $\bar{\pa}_{\CH}^*=-*{\pa}_{\CH}*$.
\end{itemize}

\begin{definition}
A compact strongly pseudoconvex complex Finsler manifold $(M,F)$ is called balanced if
the horizontal fundamental form $\oz_{\CH}$ is horizonal co-closed, i.e., $\mathrm{d}_{\CH}^*\oz_{\CH} = 0$, where $\mathrm{d}_{\CH}^*=\pa_{\CH}^* +\bar{\pa}_{\CH}^*$.
\end{definition}

By setting the horizonal torsion tensor
\beeq
S^\gz_{\az\bz}=\d{1}{2}\l(\varGamma^{\gz}_{\az;\bz}-\varGamma^{\gz}_{\bz;\az}\r),
\eneq
we introduce the following tensor \cite{LQX}
\beeq
S_\az=\sum_{\gz=1}^n S^\gz_{\az\gz}.
\eneq
Put
\[
L_{\az\bar{\bz}}^{\gz}=\d{1}{2}G^{\bar{\lz}\gz}\l[\dz_{\bar\bz}(G_{\az\bar{\lz}})
-\dz_{\bar\lz}(G_{\az\bar{\bz}})\r],
\]
then
\beeq
G^{\bar{\bz}\az}L_{\az\bar{\bz}}^{\gz}G_{\gz\bar{\lz}}
=S_{\bar\lz}=-\sum_{\gz=1}^nL_{\gz\bar{\lz}}^{\gz}.
\eneq
 By Theorem 3.5 in \cite{ZZ2}, we have
\beeq\label{24}
{\pa}_{\CH}^*\oz_{\CH}=2\sqrt{-1} S_{\bar\az} \mathrm{d}\bar{z}^{\az}{\q\rm and\q}\bar{\pa}_{\CH}^*\oz_{\CH}=-2\sqrt{-1} S_{\az} \mathrm{d}z^{\az}.
\eneq
We set a horizontal (1,0)-form $\mathcal{S}=S_{\az}\mathrm{d}z^\az$.
A direct computation shows
\[
*\CS=-\sqrt{-1}\CS\wedge \f{\oz_{\CH}^{n-1}}{(n-1)!}.
\]
Since $*\oz_{\CH}=\f{\oz_{\CH}^{n-1}}{(n-1)!}$, we have
\beeq
\pa_{\CH}\oz_{\CH}^{n-1}=-2(n-1)!\sqrt{-1}*\CS=-2\CS\wedge \oz_{\CH}^{n-1}.
\eneq
Hence,  we have
\begin{prop}
The following four statements are equivalent:
\begin{itemize}
\item[{\rm (a)}] $(M,F)$ is a balanced complex Finsler manifold;
\item[{\rm (b)}] $\CS=0${\rm(}or $S_\az=0$, or $\suml_{\gz=1}^nL_{\gz\bar{\az}}^{\gz}=0${\rm)};
\item[{\rm (c)}] $G^{\bar{\bz}\az}L_{\az\bar{\bz}}^{\gz}=0$;
\item[{\rm (d)}] $\mathrm{d}_{\CH}\oz_{\CH}^{n-1}=0$.
\end{itemize}
\end{prop}

Clearly, a K\"ahler-Finsler metric on a compact complex manifold must be balanced. When ${\rm dim}_{\BC}(M)=2$, $S_{1}=S_2=0$ means
$S^1_{12}=S^2_{12}=0$, so balanced complex Finsler surfaces are K\"ahler-Finsler.
It is natural to find a sufficient condition for a balanced complex Finsler metric to be K\"ahler-Finsler.

\subsection{Rund K\"ahler-Finsler-like metrics}

The complex Rund connection of a strongly pseudoconvex complex Finsler manifold $(M,F)$ was first introduced in \cite{Ru}, which is also a generalization of the Hermitian connection.
Denote by $\hat{D}:\CX(\CV)\rightarrow \CX(T_{\BC}^{*}{\tilde M}\otimes \CV)$ the complex Rund connection.
Its connection 1-forms are given by
\begin{equation}
\hat{\tz}_{\bz}^{\az}=G^{\bar{\lz}\az}\partial_\CH G_{\bz\bar{\lz}}
=\varGamma_{\bz;\mu}^{\az}\mathrm{d}z^{\mu}.
\end{equation}
Locally, the curvature form of the complex Rund connection $\hat{D}$ is \cite{SZ}
$$
\hat{\ooz}_\bz^\az=\hat{\ooz}^\az_{\bz;\mu\bar{\nu}}\mathrm{d}z^\mu\wedge \mathrm{d}{\bar z}^\nu+\hat{\ooz}^\az_{\bz\nu;\mu}
\mathrm{d}z^\mu\wedge \dz v^{\nu}+\hat{\ooz}^\az_{\bz\bar{\nu};\mu}\mathrm{d}z^\mu\wedge\dz\bar{v}^{\nu},
$$
where
\begin{equation}
\hat{\ooz}^\az_{\bz;\mu\bar{\nu}}=-\dz_{\bar \nu}\l(\varGamma_{\bz;\mu}^{\az}\r), \q
\hat{\ooz}^\az_{\bz\nu;\mu}= -\dot\pa_{\nu}\l(\varGamma_{\bz;\mu}^{\az}\r), \q
\hat{\ooz}^\az_{\bz\bar{\nu};\mu}=-\dot{\pa}_{\bar \nu}\l(\varGamma_{\bz;\mu}^{\az}\r).
\end{equation}

Clearly, $\Omega^{\az}_{\bz;\mu\bar{\nu}}=\hat{\Omega}^{\az}_{\bz;\mu\bar{\nu}}-\varGamma^{\az}_{\bz \gz}\dz_{\bar\nu}\l(\varGamma^{\gz}_{;\mu}\r)$. Notice that if $F$ is K\"ahler-Finsler then
$\hat{\Omega}^{\az}_{\bz;\mu\bar{\nu}}=\hat{\Omega}^{\az}_{\mu;\bz\bar{\nu}}$, but ${\Omega}^{\az}_{\bz;\mu\bar{\nu}}={\Omega}^{\az}_{\mu;\bz\bar{\nu}}$ may fail to be true.
In \cite{YZ}, the authors call a Hermitian metric $h$ is K\"ahler-like if its Chern curvature tensors $R^h$
satisfy $R^h_{X\bar{Y}Z\bar{W}}=R^h_{Z\bar{Y}X\bar{W}}$ for any type (1,0) tangent vectors $X, \,Y, \,Z$, and $W$. In \cite{Ba}, a K\"ahler-like manifold is called a Hermitian manifold with holomorphic torsion tensors. Now we extend this definition to complex Finsler geometry. We shall define a class of complex Finsler metrics that are more general than K\"ahler-Finsler metrics.

\begin{definition}
A strongly pseudoconvex complex Finsler metric $F$ is called Rund K\"ahler-Finsler-like if
the horizontal part of the curvature form of the complex Rund connection satisfies  $\hat{\ooz}^\az_{\bz;\mu\bar{\nu}}=\hat{\ooz}^\az_{\mu;\bz\bar{\nu}}$ for every $\az,\, \bz, \,
\mu, \, \nu=1,\cdots, n$.
\end{definition}

Clearly,  every K\"ahler-Finsler metric is necessarily Rund K\"ahler-Finsler-like, and $F$ is Rund K\"ahler-Finsler-like is equivalent to the equations $\dz_{\bar{\nu}}\big(S^\az_{\bz\mu}\big)=0$. A natural question is whether a Rund K\"ahler-Finsler-like metric $F$ is K\"ahler-Finsler.

\section{The canonical connection}
\setcounter{equation}{0}

Munteanu \cite{Mu} introduced a metric connection $\nabla$ with respect to the Sasaki-type lift Hermitian structure \eqref{stl} with horizontal and vertical zero torsions.
Connection coefficients of $\nabla$ are
\begin{equation}\label{ccs}
\begin{split}
L_{\bz\mu}^{\az} &=\d{1}{2}G^{\bar{\lz}\az}\l(\dz_{\mu}(G_{\bz\bar{\lz}})+\dz_{\bz}(G_{\mu\bar{\lz}})\r),\quad
C_{\bz\mu}^{\az}=\varGamma_{\bz\mu}^{\az}, \\
L_{\bz\bar{\mu}}^{\az}&=\d{1}{2}G^{\bar{\lz}\az}\l(\dz_{\bar\mu}(G_{\bz\bar{\lz}})
-\dz_{\bar\lz}(G_{\bz\bar{\mu}})\r),\quad
C_{\bz\bar\mu}^{\az}=0.
\end{split}
\end{equation}
Then connection 1-forms of $\nabla$ are given by
\begin{equation}
\oz_{\bz}^{\az}=L_{\bz\mu}^{\az}\mathrm{d}z^{\mu}+L_{\bz\bar\mu}^{\az}\mathrm{d}\bar{z}^{\mu}
+C_{\bz\mu}^{\az}\dz{v}^{\mu}.
\end{equation}
Munteanu \cite{Mu} called the connection $\nabla$ the canonical connection induced by the Chern-Finsler connection $D$. The connection $\nabla$ is a natural generalization of the complexified Levi-Civita connection on the holomorphic tangent bundle of a Hermitian manifold {\rm\cite{HLY}}.
Clearly, the connection $\nabla$ and the Chern-Finsler connection $D$ coincide if and only if $F$ is a K\"ahler-Finsler metric.

The holomorphic sectional curvature tensor of $\nabla$ is defined by
\begin{equation}\label{chsc}
R_{\az\bar\bz\mu\bar\nu}=\l\la\nabla_{\dz_{\mu}}\nabla_{\dz_{\bar\nu}}\dz_{\az}
-\nabla_{\dz_{\bar\nu}}\nabla_{\dz_{\mu}}\dz_{\az}-\nabla_{[\dz_{\mu},\dz_{\bar\nu}]}\dz_{\az},\dz_{\bz}
\r\ra_v.
\end{equation}
The horizontal holomorphic flag curvature $ K_{\nabla}(H)$ of $\nabla$ along a horizontal vector $H=H^{\az} \dz_{\az}\in\CH_v$ is defined by
\begin{equation}\label{4.2}
K_{\nabla}(H)=\d{R_{\az\bar\bz\mu\bar\nu}(v)H^{\az}\bar{H}^{\bz}H^{\mu}\bar{H}^{\nu}}{\l(G_{\az\bar\bz}(v)H^{\az}\bar{H}^{\bz}\r)^2}.
\end{equation}
The holomorphic sectional curvature $K_{\nabla}(v)$ of $\nabla$ along $v\in \tilde{M}$ is defined by
\begin{equation}\label{4.3}
K_{\nabla}(v)=\d{R_{\az\bar\bz\mu\bar\nu}(v)v^{\az}\bar{v}^{\bz}v^{\mu}\bar{v}^{\nu}}{G(v)^2}.
\end{equation}
The Ricci curvature and scalar curvature of the canonical connection $\nabla$ are
\beeq
\mathbf{Ric}_{\nabla} =\sqrt{-1}R _{\mu\bar{\nu}}\mathrm{d}z^{\mu}\wedge \mathrm{d}\bar{z}^{\nu}\q
with \q R _{\mu\bar{\nu}}=G^{\bar{\bz}\az}R_{\az\bar\bz\mu\bar\nu},
\eneq
\beeq\label{sna}
s_{\nabla} =G^{\bar{\bz}\az}G^{\bar{\nu}\mu}R_{\az\bar\bz\mu\bar\nu}.
\eneq

Now we give the explicit formula of $R_{\az\bar\bz\mu\bar\nu}$.

\begin{lem}
Let $(M,F)$ be a strongly pseudoconvex complex Finsler manifold. Then
\begin{equation}\label{hsc}
\begin{split}
&R_{\az\bar\bz\mu\bar\nu}\\
=&-\f{1}{2}\l[\dz_{\bar\nu}\dz_{\az}(G_{\mu\bar\bz})+\dz_{\bar\bz}\dz_{\mu}(G_{\az\bar\nu})
+\dz_{\bar\nu}\l(\varGamma^{\sz}_{;\mu}\r)G_{\az\bar\bz\sz}
+\dz_{\bar\bz}\l(\varGamma^{\sz}_{;\mu}\r)G_{\az\bar\nu\sz}\r]\\
&+\f{1}{4}\l[\dz_{\mu}(G_{\az\bar\lz})+\dz_{\az}(G_{\mu\bar\lz})\r]G^{\bar\lz\gz}
\l[\dz_{\bar\nu}(G_{\gz\bar\bz})+\dz_{\bar\bz}(G_{\gz\bar\nu})\r]\\
&-\f{1}{4}\big[\dz_{\bar\nu}(G_{\az\bar\lz})-\dz_{\bar\lz}(G_{\az\bar\nu})\big]G^{\bar\lz\gz}
\l[\dz_{\mu}(G_{\gz\bar\bz})-\dz_{\gz}(G_{\mu\bar\bz})\r]\\
&+\f{1}{2}\l[\dz_{\mu}\big(\varGamma^{\bar\tau}_{;\bar\bz}\big)G_{\az\bar\nu\bar\tau}
-\dz_{\mu}\l(\varGamma^{\bar\tau}_{;\bar\nu}\r)G_{\az\bar\bz\bar\tau}\r].
\end{split}
\end{equation}
\end{lem}
\begin{proof}
Since
$$
\ali
\nabla_{\dz_{\mu}}\nabla_{\dz_{\bar\nu}}\dz_{\az}
=&\l(\dz_{\mu}\l(L_{\az\bar{\nu}}^{\gz}\r)+L_{\az\bar{\nu}}^{\sz}L_{\sz\mu}^{\gz}\r)\dz_{\gz}
,\\
\nabla_{\dz_{\bar\nu}}\nabla_{\dz_{\mu}}\dz_{\az}
=&\l(\dz_{\bar\nu}\l(L_{\az\mu}^{\gz}\r)+L_{\az\mu}^{\sz}L_{\sz\bar{\nu}}^{\gz}\r)\dz_{\gz},
\eali
$$
$$
\l\la\nabla_{[\dz_{\mu},\dz_{\bar\nu}]}\dz_{\az},\dz_{\bz}\r\ra_v
=G_{\az\bar{\bz}\sz}\dz_{\bar\nu}(\varGamma^{\sz}_{;\mu}),
$$
by \eqref{chsc} we have
\beeq
R_{\az\bar\bz\mu\bar\nu}=\l[\dz_{\mu}\l(L_{\az\bar{\nu}}^{\gz}\r)
-\dz_{\bar\nu}\l(L_{\az\mu}^{\gz}\r)+L_{\az\bar{\nu}}^{\sz}L_{\sz\mu}^{\gz}
-L_{\az\mu}^{\sz}L_{\sz\bar{\nu}}^{\gz}\r]G_{\gz\bar{\bz}}
-G_{\az\bar{\bz}\sz}\dz_{\bar\nu}(\varGamma^{\sz}_{;\mu}).\label{0}
\eneq
A direct computation shows that
\beeq
\ali
&\l(\dz_{\mu}\l(L_{\az\bar{\nu}}^{\gz}\r)+L_{\az\bar{\nu}}^{\sz}L_{\sz\mu}^{\gz}\r)G_{\gz\bar\bz}\\
=&\f{1}{2}\left\{\dz_{\bar\nu}\dz_{\mu}(G_{\az\bar\bz})-\dz_{\bar\bz}\dz_{\mu}(G_{\az\bar\nu})
+\l[\dz_{\mu},\dz_{\bar\nu}\r](G_{\az\bar\bz})-\l[\dz_{\mu},\dz_{\bar\bz}\r](G_{\az\bar\nu})\right\}\\
&-\f{1}{4}\big[\dz_{\bar\nu}(G_{\az\bar\lz})-\dz_{\bar\lz}(G_{\az\bar\nu})\big]G^{\bar\lz\gz}
\big[\dz_{\mu}(G_{\gz\bar\bz})-\dz_{\gz}(G_{\mu\bar\bz})\big],\label{1}\\
\eali
\eneq
and
\beeq
\ali
&-\l(\dz_{\bar\nu}\l(L_{\az\mu}^{\gz}\r)+L_{\az\mu}^{\sz}L_{\sz\bar{\nu}}^{\gz}\r)G_{\gz\bar{\bz}}\\
=&-\f{1}{2}\l[\dz_{\bar\nu}\dz_{\mu}(G_{\az\bar\bz})+\dz_{\bar\nu}\dz_{\az}(G_{\mu\bar\bz})\r]\\
&+\f{1}{4}\l[\dz_{\mu}(G_{\az\bar\lz})+\dz_{\az}(G_{\mu\bar\lz})\r]G^{\bar\lz\gz}
\l[\dz_{\bar\nu}(G_{\gz\bar\bz})+\dz_{\bar\bz}(G_{\gz\bar\nu})\r].\label{2}
\eali
\eneq
Since
$$
\l[\dz_{\mu},\dz_{\bar\nu}\r]=\dz_{\bar\nu}\l(\varGamma^{\sz}_{;\mu}\r)\dot{\pa}_{\sz}
-\dz_{\mu}\l(\varGamma^{\bar\tau}_{;\bar\nu}\r)\dot{\pa}_{\bar\tau},
$$
we have
\beeq
\ali
&\f{1}{2}\l[\l[\dz_{\mu},\dz_{\bar\nu}\r](G_{\az\bar\bz})-\l[\dz_{\mu},\dz_{\bar\bz}\r](G_{\az\bar\nu})\r]
-G_{\az\bar{\bz}\sz}\dz_{\bar\nu}(\varGamma^{\sz}_{;\mu})\\
=&\f{1}{2}\l[-\dz_{\bar\nu}\l(\varGamma^{\sz}_{;\mu}\r)G_{\az\bar\bz\sz}
-\dz_{\mu}\l(\varGamma^{\bar\tau}_{;\bar\nu}\r)G_{\az\bar\bz\bar\tau}
-\dz_{\bar\bz}\l(\varGamma^{\sz}_{;\mu}\r)G_{\az\bar\nu\sz}
+\dz_{\mu}\big(\varGamma^{\bar\tau}_{;\bar\bz}\big)G_{\az\bar\nu\bar\tau}\r].\label{3}\\
\eali
\eneq
By substituting \eqref{1}-\eqref{3} into \eqref{0} we get \eqref{hsc}. This completes the proof.
\end{proof}
\begin{rem}\label{rem3.2}
We can check $\ov{R_{\az\bar\bz\mu\bar\nu}}=R_{\bz\bar\az\nu\bar\mu}$.
\end{rem}

Extending the product $\la\,,\ra_v$ to $T_{\BC}\tilde{M}$ and being denoted by
$$
\ali
&\big(\dz_{\az},\dz_{\bar\bz}\big)_v=\big(\dz_{\bar\bz},\dz_{\az}\big)_v=G_{\az\bar\bz},\\
&\big(\dot\pa_{\az},\dot\pa_{\bar\bz}\big)_v=\big(\dot\pa_{\bar\bz},\dot\pa_{\az}\big)_v=G_{\az\bar\bz},\\
&\big(\dz_{\az},\dz_{\bz}\big)_v=\big(\dot\pa_{\az},\dot\pa_{\bz}\big)_v=0,\\
&\big(\dz_{\az},\dot\pa_{\bz}\big)_v=\big(\dz_{\az},\dot\pa_{\bar\bz}\big)_v=0.\\
\eali
$$
Then we can extend the canonical $\nabla$ to $T_{\BC}\tilde{M}$, denoted by $\nabla^{\mathbbm{c}}$, with connection 1-forms
\begin{equation}
\ali
\oz_{\bz}^{\az}=L_{\bz\mu}^{\az}\mathrm{d}z^{\mu}+L_{\bz\bar\mu}^{\az}\mathrm{d}\bar{z}^{\mu}
+C_{\bz\mu}^{\az}\dz{v}^{\mu},\q
\oz_{\bz}^{\bar{\az}}=L_{\bz\bar\mu}^{\bar{\az}}\mathrm{d}\bar{z}^{\mu},\\
\oz_{\bar\bz}^{\az}=L_{\bar\bz\mu}^{\az}\mathrm{d}z^{\mu},\q
\oz_{\bar\bz}^{\bar{\az}}=L_{\bar\bz\mu}^{\bar\az}\mathrm{d}{z}^{\mu}+L_{\bar\bz\bar\mu}^{\bar{\az}}\mathrm{d}\bar{z}^{\mu}
+C_{\bar\bz\bar\mu}^{\bar\az}\dz\bar{v}^{\mu},
\eali
\end{equation}
where $L_{\bar\bz{\mu}}^{\az}=L_{\mu\bar{\bz}}^{\az}=\ov{L_{\bar\mu{\bz}}^{\bar\az}}
=\ov{L_{{\bz}\bar\mu}^{\bar\az}}$, $L_{\bar\bz\bar\mu}^{\bar{\az}}=\ov{L_{\bz\mu}^{\az}}$, $C_{\bar\bz\bar{\mu}}^{\bar\az}=\ov{C_{{\bz}\mu}^{\az}}$. The connection $\nabla^{\mathbbm{c}}$ is a natural generalization of the complexified Levi-Civita connection on the complexified tangent bundle of a Hermitian manifold. As \eqref{chsc}-\eqref{sna}, we can define the holomorphic sectional curvature tensor $K_{\az\bar\bz\mu\bar\nu}$, the horizontal holomorphic flag curvature ${K}_{\nabla^{\mathbbm{c}}}(H)$, the holomorphic sectional curvature ${K}_{\nabla^{\mathbbm{c}}}(v)$,  the Ricci curvature $\mathbf{Ric} _{\nabla^{\mathbbm{c}}}=\sqrt{-1}K _{\mu\bar{\nu}}\mathrm{d}z^\mu\wedge \mathrm{d}\bar{z}^{\nu}$,
the scalar curvature $s _{\nabla^{\mathbbm{c}}}$ of the connection $\nabla^{\mathbbm{c}}$.

\begin{prop}
Let $(M,F)$ be a strongly pseudoconvex complex Finsler manifold. Then
\begin{equation}\label{hsc'}
K_{\az\bar\bz\mu\bar\nu}
={R}_{\az\bar\bz\mu\bar\nu}
-\f{1}{4}\big[\dz_{\bar\bz}(G_{\mu\bar\lz})-\dz_{\bar\lz}(G_{\mu\bar\bz})\big]G^{\bar\lz\gz}
\l[\dz_{\az}(G_{\gz\bar\nu})-\dz_{\gz}(G_{\az\bar\nu})\r].
\end{equation}
Hence we have
\begin{equation}
\l(K_{\az\bar\bz\mu\bar\nu}-{R}_{\az\bar\bz\mu\bar\nu}\r)H^{\az}\bar{H}^{\bz}K^{\mu}\bar{K}^{\nu}
\leq 0, \q \forall H, \, K \in \CH.
\end{equation}
\end{prop}
\begin{proof}
A direct computation shows that
$$
\ali
\nabla^{\mathbbm{c}}_{\dz_{\mu}}\nabla^{\mathbbm{c}}_{\dz_{\bar\nu}}\dz_{\az}
=&\l[\dz_{\mu}\l(L_{\az\bar{\nu}}^{\gz}\r)+L_{\az\bar{\nu}}^{\sz}L_{\sz\mu}^{\gz}
+L_{\az\bar{\nu}}^{\bar\sz}L_{\bar{\sz}\mu}^{\gz}\r]\dz_{\gz}\\
&+\l[\dz_{\mu}\l(L_{\az\bar{\nu}}^{\bar\gz}\r)
+L_{\az\bar{\nu}}^{\bar\sz}L_{\bar{\sz}\mu}^{\bar\gz}\r]\dz_{\bar\gz},\\
\nabla^{\mathbbm{c}}_{\dz_{\bar\nu}}\nabla^{\mathbbm{c}}_{\dz_{\mu}}\dz_{\az}
=&\l[\dz_{\bar\nu}\l(L_{\az\mu}^{\gz}\r)+L_{\az\mu}^{\sz}L_{\sz\bar{\nu}}^{\gz}\r]\dz_{\gz}
+L_{\az\mu}^{\sz}L_{\sz\bar{\nu}}^{\bar\gz}\dz_{\bar\gz},
\eali
$$
$$
\l(\nabla^{\mathbbm{c}}_{[\dz_{\mu},\dz_{\bar\nu}]}\dz_{\az},\dz_{\bar\bz}\r)_v
=G_{\az\bar{\bz}\sz}\dz_{\bar\nu}(\varGamma^{\sz}_{;\mu}),
$$
we have
\beeq
\ali
K_{\az\bar\bz\mu\bar\nu}
=&\l[\dz_{\mu}\l(L_{\az\bar{\nu}}^{\gz}\r)
-\dz_{\bar\nu}\l(L_{\az\mu}^{\gz}\r)+L_{\az\bar{\nu}}^{\sz}L_{\sz\mu}^{\gz}
+L_{\az\bar{\nu}}^{\bar\sz}L_{\bar{\sz}\mu}^{\gz}\r.\\
&\l.-L_{\az\mu}^{\sz}L_{\sz\bar{\nu}}^{\gz}\r]G_{\gz\bar{\bz}}
-G_{\az\bar{\bz}\sz}\dz_{\bar\nu}(\varGamma^{\sz}_{;\mu}).
\eali
\eneq
Observe that
$$
\ali
L_{\az\bar{\nu}}^{\bar\sz}L_{\bar{\sz}\mu}^{\gz}G_{\gz\bar{\bz}}
=-\f{1}{4}\big[\dz_{\bar\bz}(G_{\mu\bar\sz})-\dz_{\bar\sz}(G_{\mu\bar\bz})\big]G^{\bar\sz\gz}
\big[\dz_{\az}(G_{\gz\bar\nu})-\dz_{\gz}(G_{\az\bar\nu})\big].\\
\eali
$$
By Lemma 3.1, the proof of \eqref{hsc'} is completed.
\end{proof}

\begin{rem}
It follows from Remark \ref{rem3.2} and \eqref{hsc'} that $\ov{K_{\az\bar\bz\mu\bar\nu}}=K_{\bz\bar\az\nu\bar\mu}$.
But $K_{\az\bar\bz\mu\bar\nu}\neq K_{\mu\bar\nu \az\bar\bz}$.
We can check that
\[
\ali
K_{\az\bar\bz\mu\bar\nu}-K_{\mu\bar\nu\az\bar\bz}
=&\f{1}{2}\l[[\dz_{\az},\dz_{\bar\bz}]\l(G_{\mu\bar\nu}\r)
-[\dz_{\mu},\dz_{\bar\nu}]\l(G_{\az\bar\bz}\r)
+[\dz_{\az},\dz_{\bar\nu}]\l(G_{\mu\bar\bz}\r)\r.\\
&\l.-[\dz_{\mu},\dz_{\bar\bz}]\l(G_{\az\bar\nu}\r)\r]
+\dz_{\az}\l(\varGamma^{\bar\tau}_{;\bar\bz}\r)G_{\mu\bar\nu\bar\tau}
-\dz_{\mu}\l(\varGamma^{\bar\tau}_{;\bar\nu}\r)G_{\az\bar\bz\bar\tau}.\\
\eali
\]
\end{rem}
%
%
%
%
\section{Curvature relations}
\setcounter{equation}{0}
%
%
%

Munteanu \cite{Mu} proved a strongly pseudoconvex complex Finsler metric $F$ is a K\"ahler-Finsler metric if and only if the canonical connection $\nabla$ and the Chern-Finsler connection $D$ coincide.
If $F$ is a K\"ahler-Finsler metric, then $R_{\az\bar\bz\mu\bar\nu}=\ooz_{\az\bar\bz;\mu\bar\nu}$. When a strongly pseudoconvex complex Finsler metric $F$ such that $R_{\az\bar\bz\mu\bar\nu}=\ooz_{\az\bar\bz;\mu\bar\nu}$, is it necessarily K\"ahler-Finsler? By virtue of Lemma 3.1, we give a positive answer.

\begin{theorem}
Let $(M,F)$ be a strongly pseudoconvex complex Finsler manifold. Then
\begin{itemize}
\item[{\rm (1)}] $R_{\az\bar\bz\mu\bar\nu}=\ooz_{\az\bar\bz;\mu\bar\nu}$, if and only if
$K_{\nabla}(H)=K_{D}(H)$, if and only if $F$ is a K\"ahler-Finsler metric;
\item[{\rm (2)}] $K_{\nabla}(v)=K_{D}(v)$,
if and only if $F$ is a weakly K\"ahler-Finsler metric.
\end{itemize}
\end{theorem}
\begin{proof}
(1) By the definition of horizontal holomorphic flag curvature,
$$
R_{\az\bar\bz\mu\bar\nu}=\ooz_{\az\bar\bz;\mu\bar\nu}
$$
is equivalent to
$$
K_{\nabla}(H)=K_{D}(H), \q\forall H=H^{\az} \dz_{\az}\in\CH_v.
$$
It follows from \eqref{hsc} in Lemma 3.1 that
$$
\ali
&R_{\az\bar\bz\mu\bar\nu}H^{\az}\bar{H}^{\bz}H^{\mu}\bar{H}^{\nu}\\
=&-\l[\dz_{\bar\nu}\dz_{\mu}(G_{\az\bar\bz})-\dz_{\mu}(G_{\az\bar\bz})G^{\bar\lz\gz}
\dz_{\bar\nu}(G_{\gz\bar\bz})+G_{\az\bar\bz\sz}\dz_{\bar\nu}\l(\varGamma^{\sz}_{;\mu}\r)
\r]H^{\az}\bar{H}^{\bz}H^{\mu}\bar{H}^{\nu}\\
&-\f{1}{2}H^{\az}\bar{H}^{\nu}\big[\dz_{\bar\nu}(G_{\az\bar\lz})-\dz_{\bar\lz}(G_{\az\bar\nu})\big]
G^{\bar\lz\gz}\l[\dz_{\mu}(G_{\gz\bar\bz})-\dz_{\gz}(G_{\mu\bar\bz})\r]\bar{H}^{\bz}H^{\mu},\\
\eali
$$
furthermore
\begin{equation}\label{rc}
\ali
&R_{\az\bar\bz\mu\bar\nu}H^{\az}\bar{H}^{\bz}H^{\mu}\bar{H}^{\nu}
-\ooz_{\az\bar\bz;\mu\bar\nu}H^{\az}\bar{H}^{\bz}H^{\mu}\bar{H}^{\nu}\\
=&-\f{1}{2}H^{\az}\bar{H}^{\nu}\big[\dz_{\bar\nu}(G_{\az\bar\lz})-\dz_{\bar\lz}(G_{\az\bar\nu})\big]G^{\bar\lz\gz}
\l[\dz_{\mu}(G_{\gz\bar\bz})-\dz_{\gz}(G_{\mu\bar\bz})\r]\bar{H}^{\bz}H^{\mu}.\\
\eali
\end{equation}
Now we use the similar argument as in the proof of Main Theorem in \cite{LQ}.
Since $(G^{\bar\lz\gz})$ is positive definite, then
$$
R_{\az\bar\bz\mu\bar\nu}H^{\az}\bar{H}^{\bz}H^{\mu}\bar{H}^{\nu}
=\ooz_{\az\bar\bz;\mu\bar\nu}H^{\az}\bar{H}^{\bz}H^{\mu}\bar{H}^{\nu},\q\forall H=H^{\az} \dz_{\az}\in\CH_v,
$$
if and only if
\begin{equation}
\l[\dz_{\mu}(G_{\gz\bar\bz})-\dz_{\gz}(G_{\mu\bar\bz})\r]\bar{H}^{\bz}H^{\mu}=0, \q\forall H=H^{\az} \dz_{\az}\in\CH_v,
\end{equation}
if and only if
$$
\dz_{\mu}(G_{\gz\bar\bz})=\dz_{\gz}(G_{\mu\bar\bz}), \q\forall 1\leq \mu, \bz,\gz \leq n,
$$
that is $F$ is a K\"ahler-Finsler metric.

(2) By \eqref{rc}, we have
\begin{equation}
\ali
&R_{\az\bar\bz\mu\bar\nu}v^{\az}\bar{v}^{\bz}v^{\mu}\bar{v}^{\nu}
-\ooz_{\az\bar\bz;\mu\bar\nu}v^{\az}\bar{v}^{\bz}v^{\mu}\bar{v}^{\nu}\\
=&-\f{1}{2}\big[\bar{\chi}(G_{\bar\lz})-G_{;\bar\lz}\big]G^{\bar\lz\gz}
\l[\chi(G_{\gz})-G_{;\gz}\r].\\
\eali
\end{equation}
Hence
$$
R_{\az\bar\bz\mu\bar\nu}v^{\az}\bar{v}^{\bz}v^{\mu}\bar{v}^{\nu}
=\ooz_{\az\bar\bz;\mu\bar\nu}v^{\az}\bar{v}^{\bz}v^{\mu}\bar{v}^{\nu},
$$
if and only if $\chi(G_{\gz})-G_{;\gz}=0$. Noting that
$$
G_{;\gz}-\chi(G_{\gz})=G_{\az}\big(\varGamma_{\bz;\gz}^{\az}-\varGamma_{\gz;\bz}^{\az}\big)v^{\bz},
$$
it follows that
$$
R_{\az\bar\bz\mu\bar\nu}v^{\az}\bar{v}^{\bz}v^{\mu}\bar{v}^{\nu}
=\ooz_{\az\bar\bz;\mu\bar\nu}v^{\az}\bar{v}^{\bz}v^{\mu}\bar{v}^{\nu}
$$
if and only if $F$ is a weakly K\"ahler-Finsler metric.
\end{proof}

\begin{rem}
By \eqref{hsc'} and \eqref{rc}, we have $K_{\az\bar\bz\mu\bar\nu}=\ooz_{\az\bar\bz;\mu\bar\nu}$ if and only if
$K_{\nabla^{\mathbbm{c}}}(H)=K_{D}(H)$, if and only if $F$ is a K\"ahler-Finsler metric.
\end{rem}

Next, we explore explicit relations between $\mathbf{Ric}_{\nabla} $ and $\mathbf{Ric}_{D}$ by using the horizontal fundamental form $\oz_{\CH}$ when $M$ is compact. As an added bonus, we obtain a sufficient and necessary condition for a compact strongly pseudoconvex complex Finsler manifold to be balanced. The following lemma follows from straightforward computations.

\begin{lem}\label{le4.2}
Let $(M,F)$ be a compact strongly pseudoconvex complex Finsler manifold. Then
\beeq
\bar{\pa}_{\CH}\bar{\pa}_{\CH}^*\oz_{\CH}
=2\sqrt{-1} \dz_{\bar\nu}\l(S_{\mu}\r)\mathrm{d}z^{\mu} \wedge \mathrm{d}\bar{z}^{\nu}\q and \q
{\pa}_{\CH}{\pa}_{\CH}^*\oz_{\CH}
=2\sqrt{-1} \dz_{\mu}\l(S_{\bar\nu}\r)\mathrm{d}z^{\mu} \wedge \mathrm{d}\bar{z}^{\nu}.
\eneq
\end{lem}

\begin{theorem}\label{th4.3}
Let $(M,F)$ be a compact strongly pseudoconvex complex Finsler manifold. Then
\begin{equation}\label{ric}
\mathbf{Ric}_{D} -\mathbf{Ric}_{\nabla} =\d{1}{2}\l({\pa}_{\CH}{\pa}_{\CH}^*\oz_{\CH}
+\bar{\pa}_{\CH}\bar{\pa}_{\CH}^*\oz_{\CH}\r),
\end{equation}
\begin{equation}\label{sca}
s_{D}-s_{\nabla}=\d{1}{2}\l\la{\pa}_{\CH}{\pa}_{\CH}^*\oz_{\CH}
+\bar{\pa}_{\CH}\bar{\pa}_{\CH}^*\oz_{\CH},\oz_{\CH}\r\ra_{\BP\tilde{M}}.
\end{equation}
Hence, the following three statements are equivalent:
\begin{itemize}
\item[{\rm (a)}] $F$ is a balanced complex Finsler metric;
\item[{\rm (b)}] $\mathbf{Ric}_{D} =\mathbf{Ric}_{\nabla} $;
\item[{\rm (c)}] $\disp\int_{\BP\tilde{M}}s_{D}\mathrm{d}\mu_{\BP\tilde{M}}
=\disp\int_{\BP\tilde{M}}s_{\nabla}\mathrm{d}\mu_{\BP\tilde{M}}$.
\end{itemize}
\end{theorem}
\begin{proof}
By \eqref{hsc}, and a straightforward computation yields
\beeq
\ali
&R _{\mu\bar\nu}\\
=&-\f{1}{2}G^{\bar{\bz}\az}\l[\dz_{\bar\nu}\dz_{\az}(G_{\mu\bar\bz})+\dz_{\mu}\dz_{\bar\bz}(G_{\az\bar\nu})
+\dz_{\bar\nu}\l(\varGamma^{\sz}_{;\mu}\r)G_{\az\bar\bz\sz}
+\dz_{\mu}\l(\varGamma^{\bar\tau}_{;\bar\nu}\r)G_{\az\bar\bz\bar\tau}\r]\\
&+\f{1}{2}G^{\bar{\bz}\az}G^{\bar\lz\gz}\l[\dz_{\mu}(G_{\az\bar\lz})\dz_{\bar\bz}(G_{\gz\bar\nu})
+\dz_{\az}(G_{\mu\bar\lz})\dz_{\bar\nu}(G_{\gz\bar\bz})\r].\\
\eali
\eneq
It is straightforward to get
\beeq\label{e}
\dz_{\mu}\l(S_{\bar\nu}\r)
=\d{1}{2}G^{\bar{\bz}\az}\l[\l(\dz_{\mu}\dz_{\bar\bz}\l(G_{\az\bar{\nu}}\r)
-\dz_{\mu}\dz_{\bar\nu}\l(G_{\az\bar{\bz}}\r)\r)-G^{\bar{\lz}\gz}\dz_{\mu}\l(G_{\az\bar{\lz}}\r)
\l(\dz_{\bar\bz}\l(G_{\gz\bar{\nu}}\r)-\dz_{\bar\nu}\l(G_{\gz\bar{\bz}}\r)\r)\r].
\eneq
It follows from \eqref{e} that
\[
\ali
\dz_{\bar\nu}\l(S_{\mu}\r)
=\d{1}{2}G^{\bar{\bz}\az}\l[\dz_{\bar\nu}\dz_{\az}\l(G_{\mu\bar{\bz}}\r)
-G^{\bar{\lz}\gz}\dz_{\az}\l(G_{\mu\bar{\lz}}\r)\dz_{\bar\nu}\l(G_{\gz\bar{\bz}}\r)\r]
-\d{1}{2}\dz_{\bar\nu}\dz_{\mu}\l(\log\det\l(G_{\az\bar{\bz}}\r)\r),
\eali
\]
\[
\ali
\dz_{\mu}\l(S_{\bar\nu}\r)
=\d{1}{2}G^{\bar{\bz}\az}\l[\dz_{\mu}\dz_{\bar\bz}\l(G_{\az\bar{\nu}}\r)
-G^{\bar{\lz}\gz}\dz_{\mu}\l(G_{\az\bar{\lz}}\r)\dz_{\bar\bz}\l(G_{\gz\bar{\nu}}\r)\r]
-\d{1}{2}\dz_{\mu}\dz_{\bar\nu}\l(\log\det\l(G_{\az\bar{\bz}}\r)\r).
\eali
\]
Hence
\begin{align}
&R _{\mu\bar\nu}+\dz_{\bar\nu}\l(S_{\mu}\r)+\dz_{\mu}\l(S_{\bar\nu}\r)\nonumber\\
=&-\d{1}{2}\dz_{\bar\nu}\dz_{\mu}\l(\log\det\l(G_{\az\bar{\bz}}\r)\r)
-\d{1}{2}\dz_{\mu}\dz_{\bar\nu}\l(\log\det\l(G_{\az\bar{\bz}}\r)\r)\nonumber\\
&-\f{1}{2}G^{\bar{\bz}\az}\l[\dz_{\bar\nu}\l(\varGamma^{\sz}_{;\mu}\r)G_{\az\bar\bz\sz}
+\dz_{\mu}\l(\varGamma^{\bar\tau}_{;\bar\nu}\r)G_{\az\bar\bz\bar\tau}\r]\nonumber\\
=&-\dz_{\bar\nu}\dz_{\mu}\l(\log\det\l(G_{\az\bar{\bz}}\r)\r)
-G^{\bar{\bz}\az}G_{\az\bar\bz\sz}\dz_{\bar\nu}\l(\varGamma^{\sz}_{;\mu}\r)\nonumber\\
=&G^{\bar{\bz}\az}\ooz_{\az\bar{\bz};\mu\bar{\nu}}\nonumber\\
=&\ooz _{\mu\bar\nu},\label{4}
\end{align}
It follows from Lemma \ref{le4.2} and \eqref{4} that equality \eqref{ric} holds.
Furthermore
\[
G^{\bar{\nu}\mu}\l(R _{\mu\bar\nu}+\dz_{\bar\nu}\l(S_{\mu}\r)+\dz_{\mu}\l(S_{\bar\nu}\r)\r)
=G^{\bar{\nu}\mu}\ooz _{\mu\bar\nu}
\]
and
\[
\d{1}{2}\l\la{\pa}_{\CH}{\pa}_{\CH}^*\oz_{\CH}
+\bar{\pa}_{\CH}\bar{\pa}_{\CH}^*\oz_{\CH},\oz_{\CH}\r\ra_{\BP\tilde{M}}
=G^{\bar{\nu}\mu}\l(\dz_{\bar\nu}\l(S_{\mu}\r)+\dz_{\mu}\l(S_{\bar\nu}\r)\r)
\]
yields equality \eqref{sca}.

Clearly, (a) implies (b), and (b) implies (c). It follows from \eqref{sca} that
\[
\ali
&\l({\pa}_{\CH}^*\oz_{\CH},{\pa}_{\CH}^*\oz_{\CH}\r)_{\BP\tilde{M}}
+\l(\bar{\pa}_{\CH}^*\oz_{\CH},\bar{\pa}_{\CH}^*\oz_{\CH}\r)_{\BP\tilde{M}}\\
=&\l({\pa}_{\CH}{\pa}_{\CH}^*\oz_{\CH}
+\bar{\pa}_{\CH}\bar{\pa}_{\CH}^*\oz_{\CH},\oz_{\CH}\r)_{\BP\tilde{M}}\\
=&\int_{\BP\tilde{M}}\l\la{\pa}_{\CH}{\pa}_{\CH}^*\oz_{\CH}
+\bar{\pa}_{\CH}\bar{\pa}_{\CH}^*\oz_{\CH},\oz_{\CH}\r\ra_{\BP\tilde{M}}\mathrm{d}\mu_{\BP\tilde{M}}\\
=&2\int_{\BP\tilde{M}}\l(s_{D}-s_{\nabla}\r)\mathrm{d}\mu_{\BP\tilde{M}}.
\eali
\]
We now assume that (c) is true, thus ${\pa}_{\CH}^*\oz_{\CH}=\bar{\pa}_{\CH}^*\oz_{\CH}=0$, i.e.,
$F$ is balanced.
\end{proof}

For simplicity, we set
\beeq
S_{\az\gz\bar{\lz}}=\d{1}{2}\l(\dz_{\gz}\l(G_{\az\bar{\lz}}\r)-\dz_{\az}\l(G_{\gz\bar{\lz}}\r)\r),
\eneq
\beeq
\mathfrak{S}=S_{\az\gz\bar{\lz}}\mathrm{d}z^\az\wedge \mathrm{d}z^\gz\wedge \mathrm{d}\bar{z}^{\nu}.\label{s}
\eneq
Now we introduce a notation
\beeq
\mathfrak{S}\circ \ov{\mathfrak{S}}:
=G^{\bar{\bz}\az}G^{\bar\lz\gz}S_{\az\gz\bar{\nu}}\ov{S_{\bz\lz\bar{\mu}}}\mathrm{d}z^\mu\wedge \mathrm{d}\bar{z}^\nu.
\eneq
Then
\beeq
G^{\bar{\nu}\mu}\l(\mathfrak{S}\circ \ov{\mathfrak{S}}\r)_{\mu\bar{\nu}}
=\l\la \sqrt{-1}\mathfrak{S}\circ \ov{\mathfrak{S}},\oz_{\CH}\r\ra_{\BP\tilde{M}}
=\l\la \mathfrak{S},\mathfrak{S}\r\ra_{\BP\tilde{M}}.
\eneq
By \eqref{hsc'}, it is easy to check that
\beeq\label{5}
\ali
K _{\mu\bar\nu}
={R} _{\mu\bar\nu}-\l(\mathfrak{S}\circ \ov{\mathfrak{S}}\r)_{\mu\bar{\nu}},
\eali
\eneq
\beeq\label{6}
s_{\nabla^{\mathbbm{c}}}
=s_{\nabla}-\l\la \mathfrak{S},\mathfrak{S}\r\ra_{\BP\tilde{M}}.
\eneq

By Theorem \ref{th4.3} and equations \eqref{5}, \eqref{6}, we obtain the following theorem immediately.

\begin{theorem}
Let $(M,F)$ be a compact strongly pseudoconvex complex Finsler manifold. Then
\begin{equation}\label{ric2}
\mathbf{Ric}_{D} -\mathbf{Ric}_{\nabla^{\mathbbm{c}}} =\d{1}{2}\l({\pa}_{\CH}{\pa}_{\CH}^*\oz_{\CH}
+\bar{\pa}_{\CH}\bar{\pa}_{\CH}^*\oz_{\CH}\r)+\sqrt{-1}\mathfrak{S}\circ \ov{\mathfrak{S}},
\end{equation}
\begin{equation}\label{sca2}
s_{D}-s_{\nabla^{\mathbbm{c}}}=\d{1}{2}\l\la{\pa}_{\CH}{\pa}_{\CH}^*\oz_{\CH}
+\bar{\pa}_{\CH}\bar{\pa}_{\CH}^*\oz_{\CH},\oz_{\CH}\r\ra_{\BP\tilde{M}}
+\la \mathfrak{S},\mathfrak{S}\ra_{\BP\tilde{M}}.
\end{equation}
Hence, the following four statements are equivalent:
\begin{itemize}
\item[{\rm (a)}] $F$ is a K\"ahler-Finsler metric;
\item[{\rm (b)}] $\mathbf{Ric}_{D} =\mathbf{Ric}_{\nabla^{\mathbbm{c}} } $;
\item[{\rm (c)}] $\disp\int_{\BP\tilde{M}}s_{\nabla}\mathrm{d}\mu_{\BP\tilde{M}}
=\disp\int_{\BP\tilde{M}}s_{\nabla^{\mathbbm{c}}}\mathrm{d}\mu_{\BP\tilde{M}}$;
\item[{\rm (d)}] $\disp\int_{\BP\tilde{M}}s_{D}\mathrm{d}\mu_{\BP\tilde{M}}
=\disp\int_{\BP\tilde{M}}s_{\nabla^{\mathbbm{c}}}\mathrm{d}\mu_{\BP\tilde{M}}$.
\end{itemize}
\end{theorem}

\section{A sufficient condition for balanced complex Finsler manifolds to be K\"ahler-Finsler}
\setcounter{equation}{0}

First, we response the question raised in subsection 2.3 when the complex manifold is compact.
Now we give a sufficient condition for a compact Rund K\"ahler-Finsler-like manifold to be K\"ahler-Finsler.

\begin{theorem}
Let $(M,F)$ be a compact Rund K\"ahler-Finsler-like manifold. If there is a positive horizontal $(n-2,n-2)$ form $\varphi$ on $\BP\tilde{M}$ such that $\pa_{\CH}\bar{\pa}_{\CH}\varphi=0$,
then $F$ is a K\"ahler-Finsler metric. In particular, if $\pa_{\CH}\bar{\pa}_{\CH}\oz_{\CH}^{n-2}=0$, then
$F$ is a K\"ahler-Finsler metric.
\end{theorem}
\begin{proof}
A direct computation shows
\beeq\label{ddbar}
\ali
\pa_{\CH}\bar{\pa}_{\CH}\oz_{\CH}=&\frac{\sqrt{-1}}{4}\l[\dz_{\mu}\dz_{\bar\bz}\l(G_{\az\bar{\nu}}\r)
+\dz_{\az}\dz_{\bar\nu}\l(G_{\mu\bar{\bz}}\r)-\dz_{\az}\dz_{\bar\bz}\l(G_{\mu\bar{\nu}}\r)
-\dz_{\mu}\dz_{\bar\nu}\l(G_{\az\bar{\bz}}\r)\r]\\
& \mathrm{d}z^\az\wedge \mathrm{d}z^\mu \wedge \mathrm{d}\bar{z}^\bz \wedge \mathrm{d}\bar{z}^\nu.
\eali
\eneq
When $F$ is Rund K\"ahler-Finsler-like, a routine calculation implies
\[
\ali
&\l[\dz_{\mu}\dz_{\bar\bz}\l(G_{\az\bar{\nu}}\r)-\dz_{\mu}\dz_{\bar\nu}\l(G_{\az\bar{\bz}}\r)\r]
+\l[\dz_{\az}\dz_{\bar\nu}\l(G_{\mu\bar{\bz}}\r)-\dz_{\az}\dz_{\bar\bz}\l(G_{\mu\bar{\nu}}\r)\r]\\
=&\dz_{\mu}\l[\l(\varGamma^{\bar{\lz}}_{\bar{\nu};\bar{\bz}}
-\varGamma^{\bar{\lz}}_{\bar{\bz};\bar{\nu}}\r)G_{\az\bar{\lz}}\r]
+\dz_{\az}\l[\l(\varGamma^{\bar{\lz}}_{\bar{\bz};\bar{\nu}}
-\varGamma^{\bar{\lz}}_{\bar{\nu};\bar{\bz}}\r)G_{\mu\bar{\lz}}\r]\\
=&2\l[\dz_{\az}\l(G_{\mu\bar{\lz}}\r)-\dz_{\mu}\l(G_{\az\bar{\lz}}\r)\r]
\ov{S^{\lz}_{\bz\nu}}\\
=&-4G_{\gz\bar{\lz}}S^{\gz}_{\az\mu}\ov{S^{\lz}_{\bz\nu}}.\\
\eali
\]
It follows from \eqref{ddbar} that
\beeq\label{ddbar2}
\sqrt{-1}\pa_{\CH}\bar{\pa}_{\CH}\oz_{\CH}= G_{\gz\bar{\lz}}S^{\gz}_{\az\mu}\ov{S^{\lz}_{\bz\nu}}
 \mathrm{d}z^\az\wedge \mathrm{d}z^\mu \wedge \mathrm{d}\bar{z}^\bz \wedge \mathrm{d}\bar{z}^\nu.
\eneq
Notice that $G_{\gz\bar{\lz}}S^{\gz}_{\az\mu}\ov{S^{\lz}_{\bz\nu}}
 \mathrm{d}z^\az\wedge \mathrm{d}z^\mu \wedge \mathrm{d}\bar{z}^\bz \wedge \mathrm{d}\bar{z}^\nu$ is a global nonnegative horizontal $(2,2)$-form defined on $\BP\tilde{M}$
 and vanishes if and only if $F$ is horizonal torsion-free.
Suppose $M$ is compact and there is a positive horizontal form $\varphi\in \CA^{n-2,n-2}$ such that $\pa_{\CH}\bar{\pa}_{\CH}\varphi=0$.
Then
\[
\ali
&(\sqrt{-1}\oz_{\CH},\bar{\pa}_{\CH}^{*}\pa_{\CH}^{*}*\varphi)_{\BP\tilde{M}}\\
=&(\sqrt{-1}\pa_{\CH}\bar{\pa}_{\CH}\oz_{\CH},*\varphi)_{\BP\tilde{M}}\\
=&\int_{\BP\tilde{M}}\sqrt{-1}\pa_{\CH}\bar{\pa}_{\CH}\oz_{\CH}\wedge \varphi\wedge \d{\oz_{\CV}^{n-1}}{(n-1)!} .
\eali
\]
On the other hand,
\[
\bar{\pa}_{\CH}^{*}\pa_{\CH}^{*}*\varphi=*\pa_{\CH}**\bar{\pa}_{\CH}**\varphi
=-*\pa_{\CH}\bar{\pa}_{\CH}\varphi=0.
\]
Hence, $\sqrt{-1}\pa_{\CH}\bar{\pa}_{\CH}\oz_{\CH}=0$. It follows from \eqref{ddbar2}
that $S^{\gz}_{\az\mu}=0$, or equivalently, $F$ is a K\"ahler-Finsler metric.
\end{proof}

From the above theorem, we can see every compact Rund K\"ahler-Finsler-like surface must be K\"ahler-Finsler. Now we explore the relation between Rund K\"ahler-Finsler-like metrics
and balanced complex Finsler metrics on a compact complex manifold.

\begin{theorem}
Let $(M,F)$ be a Rund K\"ahler-Finsler-like manifold. If $M$ is compact, then $F$ is a balanced
complex Finsler metric.
\end{theorem}
\begin{proof}
Suppose $M$ is compact. Let $X=X^\az\dz_\az$ be a horizontal vector field on $\BP\tilde{M}$.
We define the divergence of $X $ with respect to $\mathrm{d}\mu_{\BP\tilde{M}}$ by
\begin{equation}
({\rm div} X)\mathrm{d}\mu_{\BP\tilde{M}}=\CL_X \mathrm{d}\mu_{\BP\tilde{M}},
\end{equation}
where $\CL_X$ is the Lie derivative along $X$ with respect to the volume form $\mathrm{d}\mu_{\BP\tilde{M}}$.
We denote by $\mathfrak{i}\l(X\r)$ the inter product operator.
By \cite{ZhC}, we have
\begin{equation}\label{stokes}
({\rm div} X)\mathrm{d}\mu_{\BP\tilde{M}}=\mathrm{d}\l(\mathfrak{i}\l(X\r) \mathrm{d}\mu_{\BP\tilde{M}}\r),
\end{equation}
\begin{equation}\label{div}
{\rm div} X=\dz_{\az}\l(X^\az\r)+X^\az\varGamma^{o}_{o;\az}.
\end{equation}
We define a horizontal (1,0) vector field
\beeq
\uparrow \mathcal{S}=G^{\bar{\nu}\az}S_{\bar\nu}\dz_{\az}.
\eneq
By \eqref{div}, we have
\beeq\label{7}
{\rm div}\l(\uparrow \mathcal{S}\r)=G^{\bar{\nu}\az}\dz_{\az}\l(S_{\bar{\nu}}\r)
-2G^{\bar{\nu}\az}S_{\az} S_{\bar\nu}.
\eneq
Integrating on both sides of  equality \eqref{7} on $\BP\tilde{M}$ together with \eqref{stokes} yields
\[
\int_{\BP\tilde{M}}G^{\bar{\nu}\az}\l[\dz_{\az}\l(S_{\bar{\nu}}\r)
-2S_{\az} S_{\bar\nu}\r]\mathrm{d}\mu_{\BP\tilde{M}}=0.
\]
When $F$ is Rund K\"ahler-Finsler-like, we have $\dz_{\az}\big(S_{\bar{\nu}}\big)=0$ and
$(\mathcal{S},\mathcal{S})_{\BP\tilde{M}}=0$. Thus every $S_{\az}=0$, that is $F$ is balanced.
\end{proof}

\begin{rem}
From the above proof, we can see a compact strongly pseudoconvex complex Finsler manifold is balanced if and only if $\bar{\pa}_{\CH}\CS=0$.
\end{rem}

Next, we response the question raised in subsection 2.2 when the complex manifold is compact. We obtain the following result:

\begin{theorem}
Let $(M,F)$ be a balanced complex Finsler manifold. Then $F$ is a K\"ahler-Finsler metric, if and only if
$$
\l\la\pa_{\CH}\bar{\pa}_{\CH}\oz_{\CH}, \oz_{\CH}^2\r\ra_{\BP\tilde{M}}=0.
$$
\end{theorem}

\begin{proof}
We need only to prove the sufficiency.
By \eqref{ddbar}, we have
\[
\l\la\pa_{\CH}\bar{\pa}_{\CH}\oz_{\CH}, \oz_{\CH}^2\r\ra_{\BP\tilde{M}}
=\sqrt{-1}G^{\bar{\bz}\az}G^{\bar{\nu}\mu}\l[\dz_{\mu}\dz_{\bar\bz}\l(G_{\az\bar{\nu}}\r)
+\dz_{\az}\dz_{\bar\nu}\l(G_{\mu\bar{\bz}}\r)-\dz_{\az}\dz_{\bar\bz}\l(G_{\mu\bar{\nu}}\r)
-\dz_{\mu}\dz_{\bar\nu}\l(G_{\az\bar{\bz}}\r)\r].
\]
Then the condition $\l\la\pa_{\CH}\bar{\pa}_{\CH}\oz_{\CH}, \oz_{\CH}^2\r\ra_{\BP\tilde{M}}=0$ is equivalent to
\begin{equation}\label{d}
G^{\bar{\bz}\az}G^{\bar{\nu}\mu}\l[\dz_{\mu}\dz_{\bar\bz}\l(G_{\az\bar{\nu}}\r)
+\dz_{\az}\dz_{\bar\nu}\l(G_{\mu\bar{\bz}}\r)-\dz_{\az}\dz_{\bar\bz}\l(G_{\mu\bar{\nu}}\r)
-\dz_{\mu}\dz_{\bar\nu}\l(G_{\az\bar{\bz}}\r)\r]=0.
\end{equation}
By the balanced condition, we have
\beeq
G^{\bar{\bz}\az}\l[\dz_{\mu}\dz_{\bar\bz}\l(G_{\az\bar{\nu}}\r)
-\dz_{\mu}\dz_{\bar\nu}\l(G_{\az\bar{\bz}}\r)\r]
=2G^{\bar{\bz}\az}G^{\bar{\lz}\gz}\dz_{\mu}\l(G_{\az\bar{\lz}}\r)\ov{S_{\nu\bz\bar{\gz}}}.
\eneq
Exchanging indices $\az,\,\bar{\bz}$ with $\mu,\,\bar{\nu}$, respectively, implies
\beeq
G^{\bar{\nu}\mu}\l[\dz_{\az}\dz_{\bar\nu}\l(G_{\mu\bar{\bz}}\r)
-\dz_{\az}\dz_{\bar\bz}\l(G_{\mu\bar{\nu}}\r)
\r]=2G^{\bar{\nu}\mu}\dz_{\az}\l(G_{\mu\bar{\lz}}\r)G^{\bar{\lz}\gz}\ov{S_{\bz\nu\bar{\gz}}}.
\eneq
Therefore
$$
\ali
&G^{\bar{\bz}\az}G^{\bar{\nu}\mu}\l[\dz_{\mu}\dz_{\bar\bz}\l(G_{\az\bar{\nu}}\r)
+\dz_{\az}\dz_{\bar\nu}\l(G_{\mu\bar{\bz}}\r)-\dz_{\az}\dz_{\bar\bz}\l(G_{\mu\bar{\nu}}\r)
-\dz_{\mu}\dz_{\bar\nu}\l(G_{\az\bar{\bz}}\r)\r]\\
=&-4G^{\bar{\bz}\az}G^{\bar{\nu}\mu}G^{\bar{\lz}\gz}S_{\az\mu\bar{\lz}}\ov{S_{\bz\nu\bar{\gz}}}
=-4\la \mathfrak{S},\mathfrak{S}\ra_{\BP\tilde{M}}.
\eali
$$
By \eqref{d}, we have
$$
\l(\mathfrak{S},\mathfrak{S}\r)_{\BP\tilde{M}}
=\disp\int_{\BP\tilde{M}}\la \mathfrak{S},\mathfrak{S}\ra_{\BP\tilde{M}}\mathrm{d}\mu_{\BP\tilde{M}}=0
$$
Hence, $S_{\az\mu\bar{\lz}}=0$ and $F$ is a K\"ahler-Finsler metric.
\end{proof}

Now we give two examples of balanced complex Finsler metrics (resp. Rund K\"ahler-Finsler-like metrics).

\begin{ex} {\rm(Szab$\acute{o}$ metric \cite{CY, XZ4, XZ5})} \quad
Let $M_1,~M_2$ be two complex manifolds with ${\rm dim}_{\BC}M_1=n$ and ${\rm dim}_{\BC}M_2=m$, and $z=(z_1, z_2)\in M_1\times M_2$
and $v=(v_1,v_2)\in T^{1,0}_{z_1}M_1\oplus T^{1,0}_{z_2}M_2$ be the local complex coordinates on $M_1\times M_2$ and the fiber coordinates on
$T^{1,0}_{(z_1, z_2)}(M_1\times M_2)$, respectively. One can define the Szab$\acute{o}$ metric on $M_1\times M_2$ by
\beeq
F_{\varepsilon}=\sqrt{h_1^2(z_1, {v}_1)+h_2^2(z_2, {v}_2)
+\varepsilon\l(\l(h_1(z_1, {v}_1)\r)^{2k}+\l(h_2(z_2, {v}_2)\r)^{2k}\r)^{\frac{1}{k}}},
\eneq
where $\varepsilon>0$, $k>0$, are constants,
and $h_1$ and $h_2$ are Hermitian metrics on $M_1$ and $M_2$, respectively.
\end{ex}

It was proved that $F_{\varepsilon}$ is a strongly pseudoconvex metric on $M_1\times M_2$, and that $F_{\varepsilon}$ is a K\"ahler-Finsler metric if and only if $h_1$ and $h_2$ are both K\"ahler metrics on $M_1$ and $M_2$, respectively. We denote by ${^{M_1}\varGamma}^{\az}_{\bz;\mu} $ and ${^{M_2}\varGamma}^{i}_{j;k}$
the Chern connection coefficients of $h_1$ and $h_2$, respectively, where lowercase Greek indices run from 1 to $n$ and lowercase Latin indices run from 1 to $m$. The horizontal connection coefficients of the Chern-Finsler connection associated to $F_\varepsilon$ are \cite{CY}
\beeq
\ali
&\q\varGamma^{\az}_{\bz;\mu}={^{M_1}\varGamma}^{\az}_{\bz;\mu},\q
\varGamma^{i+n}_{j+n;k+n}={^{M_2}\varGamma}^{i}_{j;k},\\
\varGamma^{\az}_{j+n;\mu}=&\varGamma^{\az}_{\bz;k+n}=\varGamma^{\az}_{j+n;k+n}
=\varGamma^{i+n}_{\bz;\mu}=\varGamma^{i+n}_{j+n;\mu}=\varGamma^{i+n}_{\bz;k+n}=0.
\eali
\eneq
Hence, $S_\az={^{M_1}S}_{\az}$, $S_{i+n}={^{M_2}S}_{i}$.
Notice that the horizontal connection coefficients of the Chern-Finsler connection associated to $F_\varepsilon$ depend only on the base manifold coordinates, we have
\[
\hat{\ooz}^{\az}_{\bz;\mu\bar{\nu}}=-\d{\pa}{\pa \bar{z}^{\nu}}\l({^{M_1}\Gamma}^{\az}_{\bz;\mu}\r),\q
\hat{\ooz}^{i+n}_{j+n;k+n\, \ov{l+n} }=-\d{\pa}{\pa \bar{z}^{l+n}}\l({^{M_2}\Gamma}^{i}_{j;k}\r),
\]
and the other horizontal curvature terms of the Rund connection are zero. Hence
\begin{prop}
The symbols put as above. Suppose that $M_1$ and $M_2$ are both compact complex manifold. Then

{\rm(i)} $F_{\varepsilon}$ is a balanced
complex Finsler metric if and only if $h_1$ and $h_2$ are balanced metrics on $M_1$ and $M_2$ , respectively.

{\rm(ii)} $F_{\varepsilon}$ is a Rund K\"ahler-Finsler-like metric if and only if $h_1$ and $h_2$ are K\"ahler-like metrics on $M_1$ and $M_2$, respectively.
\end{prop}

By analogy with the Randers metrics in real Finsler geometry,
Aldea and Munteanu introduced complex Randers metrics.

\begin{ex} {\rm(complex Randers metric \cite{AM, CS2, XZ6})} \quad
Let $\az(z,v)=\sqrt{h_{\az\bar \bz}(z)v^\az \bar{v}^{\bz}}$
be a Hermitian metric on a complex manifold $M$ and $\bz(z,v)=b_\az(z)v^\az$
be a $(1,0)$-form. The complex Randers metric on $M$ are defined by
 $$
 F(z,v)=\az+|\bz|.
 $$
\end{ex}

It was proved that $F(z,v)=\az+|\bz|$ is a strongly pseudoconvex (even strongly convex) complex Finsler metric.
Moreover, if $\bz\otimes \bar{\bz}$ is parallel with respect to $\az$, then
$\varGamma^{\az}_{\bz;\mu}=\gz^{\az}_{\bz;\mu}$, where $\gz^{\az}_{\bz;\mu}$ are the Chern connection coefficients of $\az$.
Hence
\begin{prop}
The symbols put as above. Suppose that  $\bz\otimes \bar{\bz}$ is parallel with respect to $\az$. Then

{\rm(i)}  $F=\az+|\bz|$ is a balanced complex Finsler metric if and only if $\az$ is a balanced metric.

{\rm(ii)} $F=\az+|\bz|$ is a Rund K\"ahler-Finsler-like metric if and only if $\az$ is a K\"ahler-like metric.
\end{prop}

\section{Conformal transformations of a balanced complex Finsler manifold}
\setcounter{equation}{0}

Suppose that $\tilde {F}(z,v)={\rm e}^{\frac{1}{2}\rz(z)}F(z,v)$ is a conformal metric of a strongly pseudoconvex complex Finsler metric $F$, where $\rz$ is a smooth real function on $M$.
Set $\tilde {G}=\tilde {F}^2$ and $\tilde {G}_{\az\bar{\bz}}
=\dot\pa_\az\dot\pa_{\bar \bz}(\tilde {G})$.
From now on, we denote by an extra symbol `` $\sim$ " the corresponding quantities with respect to the conformal metric $\tilde {F}(z,v)$.

\begin{lem}{\rm\cite{Al1}}\label{lem6.1}
Suppose $\tilde {F}(z,v)={\rm e}^{\frac{1}{2}\rz(z)}F(z,v)$ is a conformal metric of
a strongly pseudoconvex complex Finsler metric $F$. Then
\begin{align*}
 \tilde{\varGamma}_{;\bz}^{\az}&=\varGamma_{;\bz}^{\az}+\rz_{;\bz}v^\az,
  \q~ \tilde{\dz}_\bz=\dz_\bz-\rz_{;\bz}\iota,\\
\tilde \varGamma_{\bz;\gz}^{\az}&=\varGamma_{\bz;\gz}^{\az}+\dz^\az_\bz\rho_{;\gz},
\q \tilde \varGamma_{\bz\gz}^{\az}=\varGamma_{\bz\gz}^{\az},
\end{align*}
where $\rho_{;\bz}=\frac{\pa \rz}{\pa z^\bz}$, and $\iota=v^\az\dot\pa_\az$ is the radial vertical vector field associated to $F$.
\end{lem}

\begin{theorem}
On a compact complex manifold $M$, there exists at most one balanced complex Finsler metric (up to constant multiples) in each conformal class of complex Finsler metrics. If  $(M,F)$ is a non-compact Rund K\"ahler-Finsler-like manifold, then the conformal 
metric $\tilde {F}(z,v)={\rm e}^{\frac{1}{2}\rz(z)}F(z,v)$ is Rund K\"ahler-Finsler-like if and only if $\pa\bar{\pa}\rho=0$.
\end{theorem}

\begin{proof}
Suppose that $F$ is a balanced complex Finsler metric on a compact complex manifold, and
$\tilde {F}(z,v)={\rm e}^{\frac{1}{2}\rz(z)}F(z,v)$ is a conformal transformation of $F$. By equality (3.4) in \cite{LQX}, we have
\beeq\label{6.1}
\tilde{S}_{\az}=S_{\az}+\d{1}{2}(1-n)\rho_{;\az}.
\eneq
If $\tilde{F}$ is also balanced, then $\tilde{S}_{\az}=S_{\az}=0$ and $\rho_{;\az}=0$,
i.e., $\rho$ is a constant.

Now assume that $(M, F)$ is a non-compact strongly pseudoconvex complex Finsler manifold.
Noting that ${S}^{\az}_{\bz\gz}$ is $(0, 0)$ homogeneous with respect to the fiber coordinate $v$, we have $\iota({S}^{\az}_{\bz\gz})=0$.
By Lemma \ref{lem6.1}, we have
\beeq
\tilde{\dz}_{\bar{\nu}}\l(\tilde{S}^{\az}_{\bz\gz}\r)
=\dz_{\bar\nu}\l(S_{\bz\gz}^{\az}\r)
+\frac{1}{2}(\dz^\az_\bz\rho_{;\gz\bar{\nu}}-\dz^\az_\gz\rho_{;\bz\bar{\nu}}).
\eneq
If $F$ is Rund K\"ahler-Finsler-like, then the conformal metric  $\tilde{F}$ is Rund K\"ahler-Finsler-like if and only if 
$\dz^\az_\bz\rho_{;\gz\bar{\nu}}-\dz^\az_\gz\rho_{;\bz\bar{\nu}}=0$.
Contracting on indices $\az$ and $\bz$, it yields $(n-1)\rho_{;\gz\bar{\nu}} = 0$, which is equivalent to
$\pa\bar{\pa}\rho=0$. This completes the proof.
\end{proof}

\begin{rem}
Obviously, the equality $\pa\bar{\pa}\rho=0$ implies that $\rho$ is a harmonic function, however, the converse is not ture. Especially for ${\rm dim}_{\BC}M=2$, it is easy to check that $\rho_1(z_1, z_2)=x_1x_2-y_1y_2$ and
 $\rho_2(z_1, z_2)=x_1y_2+y_1x_2$ both satisfy $\pa\bar{\pa}\rho=0$.
\end{rem}

\begin{theorem}
On a compact strongly pseudoconvex complex Finsler manifold $(M,F)$, the conformal metric $\tilde {F}(z,v)={\rm e}^{\frac{1}{2}\rz(z)}F(z,v)$ is balanced, if and only if
\beeq\label{6.3}
\mathbf{Ric}_{D} -\mathbf{Ric}_{\nabla} =(n-1)\sqrt{-1}\pa\bar{\pa}\rho,
\eneq
if and only if
\beeq\label{6.4}
\int_{\BP\tilde{M}}\l(s_{D} -s_{\nabla} \r)\mathrm{d}\mu_{\BP\tilde{M}}
=(n-1)\l(\sqrt{-1}\pa\bar{\pa}\rho,\oz_{\CH}\r)_{\BP\tilde{M}}.
\eneq
\end{theorem}

\begin{proof}
By Lemma \ref{lem6.1} and equalitis \eqref{24}, \eqref{6.1}, we obtain
\beeq
\bar{\pa}_{\widetilde{\CH}}^*\oz_{\widetilde{\CH}}
=\bar{\pa}_{\CH}^*\oz_{\CH}+(n-1)\sqrt{-1}\pa\rho.
\eneq
Therefore
\beeq\label{6.6}
\bar{\pa}_{\widetilde{\CH}}\bar{\pa}_{\widetilde{\CH}}^*\oz_{\widetilde{\CH}}
=\bar{\pa}_{\CH}\bar{\pa}_{\CH}^*\oz_{\CH}-(n-1)\sqrt{-1}\pa\bar{\pa}\rho
\eneq
and
\beeq
{\pa}_{\widetilde{\CH}}\bar{\pa}_{\widetilde{\CH}}^*\oz_{\widetilde{\CH}}
={\pa}_{\CH}\bar{\pa}_{\CH}^*\oz_{\CH}.
\eneq
If $\tilde{F}$ is balanced, i.e., $\bar{\pa}_{\widetilde{\CH}}^*\oz_{\widetilde{\CH}}=0$,
then
\beeq\label{25}
{\pa}_{\CH}{\pa}_{\CH}^*\oz_{\CH}+\bar{\pa}_{\CH}\bar{\pa}_{\CH}^*\oz_{\CH}=2(n-1)\sqrt{-1}\pa\bar{\pa}\rho.
\eneq
It follows from \eqref{ric} and \eqref{25} that \eqref{6.3} holds. On the other hand, suppose that \eqref{6.3} holds.
By \eqref{ric} and\eqref{6.6}, we obtain ${\pa}_{\widetilde{\CH}}{\pa}_{\widetilde{\CH}}^*\oz_{\widetilde{\CH}}
+\bar{\pa}_{\widetilde{\CH}}\bar{\pa}_{\widetilde{\CH}}^*\oz_{\widetilde{\CH}}=0$.
Compactness can reduce ${\pa}_{\widetilde{\CH}}^*\oz_{\widetilde{\CH}}
=\bar{\pa}_{\widetilde{\CH}}^*\oz_{\widetilde{\CH}}=0$.
This completes the proof.
\end{proof}
\begin{theorem}
Suppose that $(M,F)$ is a balanced complex Finsler manifold with positive Chern-Finsler scalar curvature. Then the conformal metric $\tilde {F}(z,v)={\rm e}^{\frac{1}{2}\rz(z)}F(z,v)$ owns positive Chern-Finsler total scalar curvature, that is
\beeq\label{ }
\int_{\BP\tilde{M}}\tilde{s}_{\tilde{D}}  \mathrm{d}\tilde{\mu}_{\BP\tilde{M}}>0.
\eneq
\end{theorem}
\begin{proof}
By Lemma \ref{lem6.1}, a direct computation shows
\beeq
\tilde{\ooz}_{\az\bar{\bz}}={\ooz}_{\az\bar{\bz}}-n\rho_{;\az\bar{\bz}},
\eneq
and then
\beeq
\tilde{s}_{\tilde{D}}={\rm e}^{-\rho}\l(s_D-nG^{\bar{\bz}\az}\rho_{;\az\bar{\bz}}\r).
\eneq
Note that \cite{CSZ}
\[
\mathrm{d}\tilde{\mu}_{\BP\tilde{M}}={\rm e}^{n\rho}\mathrm{d}{\mu}_{\BP\tilde{M}},
\]
thus
\beeq
\int_{\BP\tilde{M}}\tilde{s}_{\tilde{D}}  \mathrm{d}\tilde{\mu}_{\BP\tilde{M}}
=\int_{\BP\tilde{M}}{\rm e}^{(n-1)\rho}\l(s_D-nG^{\bar{\bz}\az}\rho_{;\az\bar{\bz}}\r) \mathrm{d}{\mu}_{\BP\tilde{M}}.
\eneq
By the balanced condition, we have
\beeq
{\rm div}\l(\uparrow \pa{\rm e}^{(n-1)\rho}\r)=G^{\bar{\bz}\az}({\rm e}^{(n-1)\rho})_{;\az\bar{\bz}}.
\eneq
Hence
\[
\ali
&n\int_{\BP\tilde{M}}{\rm e}^{(n-1)\rho}G^{\bar{\bz}\az}\rho_{;\az\bar{\bz}} \mathrm{d}{\mu}_{\BP\tilde{M}}\\
=&\d{n}{n-1}\int_{\BP\tilde{M}}G^{\bar{\bz}\az}\l[\l({\rm e}^{(n-1)\rho}\r)_{;\az\bar{\bz}}
-(n-1)^2{\rm e}^{(n-1)\rho}\rho_{;\az}\rho_{;\bar{\bz}}\r]\mathrm{d}{\mu}_{\BP\tilde{M}}\\
=&\d{n}{n-1}\int_{\BP\tilde{M}}\l[{\rm div}\l(\uparrow \pa{\rm e}^{(n-1)\rho}\r)
-4\l\la\pa{\rm e}^{\f{1}{2}(n-1)\rho},\pa{\rm e}^{\f{1}{2}(n-1)\rho}\r\ra_{\BP\tilde{M}}\r]
\mathrm{d}{\mu}_{\BP\tilde{M}}\\
=&-\d{4n}{n-1}\l(\pa{\rm e}^{\f{1}{2}(n-1)\rho}, \pa{\rm e}^{\f{1}{2}(n-1)\rho}\r)_{\BP\tilde{M}}.
\eali
\]
Therefore, we obtain
\beeq
\int_{\BP\tilde{M}}\tilde{s}_{\tilde{D}}  \mathrm{d}\tilde{\mu}_{\BP\tilde{M}}
=\int_{\BP\tilde{M}}{\rm e}^{(n-1)\rho}s_D \mathrm{d}{\mu}_{\BP\tilde{M}}+\d{4n}{n-1}\l(\pa{\rm e}^{\f{1}{2}(n-1)\rho},\pa{\rm e}^{\f{1}{2}(n-1)\rho}\r)_{\BP\tilde{M}}.
\eneq
If $s_D>0$, then $\disp\int_{\BP\tilde{M}}\tilde{s}_{\tilde{D}}  \mathrm{d}\tilde{\mu}_{\BP\tilde{M}}
\geq\int_{\BP\tilde{M}}{\rm e}^{(n-1)\rho}s_D \mathrm{d}{\mu}_{\BP\tilde{M}}>0$.
\end{proof}

\bigskip

\noindent{\small{\bf Acknowledgements:}\
This research is supported by the National Natural Science Foundation of China (Grant Nos. 12001165, 12071386, 11701494),  the Nanhu Scholars Program for Young Scholars of Xinyang Normal University and the Key Research Project of Henan Higher Education Institutions(China) (No. 22A110021).}

\end{document}